\documentclass[reqno]{amsart}%
\usepackage[utf8]{inputenc}%
\usepackage[english]{babel}%
\usepackage{amsmath,amssymb,amsthm,amsfonts}%
\usepackage{hyperref}%
\usepackage{enumerate}%
\usepackage{graphicx}
\usepackage{mathrsfs}
\usepackage{color}
\usepackage{accents}
\allowdisplaybreaks%
\numberwithin{equation}{section}%

\usepackage{multicol}
\newcommand{\Z}{\mathbb{Z}}
\newcommand{\C}{\mathbb{C}}
\newcommand{\R}{\mathbb{R}}

\renewcommand{\i}{\mathrm{i}}

\newcommand{\al}{\alpha}

\newcommand{\la}{\lambda}
\newcommand{\La}{\Lambda}

\newcommand{\be}{\beta}

\newcommand{\de}{\delta}
\newcommand{\ga}{\gamma}
\newcommand{\ka}{\varkappa}

\newcommand{\GT}{\mathbb{GT}}
\newcommand{\T}{\mathbb{T}}
\newcommand{\Ps}{\mathsf{P}}
\newcommand{\Vs}{\mathbf{V}}
\newcommand{\x}{\mathsf{x}}

\newcommand{\Om}{\Omega}
\newcommand{\om}{\omega}
\DeclareMathOperator{\Dim}{\mathrm{Dim}}
\newcommand{\Cont}{\mathfrak{C}}
\newcommand{\cont}{\mathfrak{c}}
\newcommand{\q}{{}_{q}}
\newcommand{\qi}{{}_{q^{-1}}}
\newcommand{\vol}{\mathsf{vol}}
\newcommand{\Lb}{\mathbb{L}}
\newcommand{\Rs}{\mathcal{R}}

\newtheorem{proposition}{Proposition}[section]
\newtheorem{lemma}[proposition]{Lemma}

\newtheorem{theorem}[proposition]{Theorem}
\theoremstyle{definition}
\newtheorem{question}[proposition]{Question}
\newtheorem{definition}[proposition]{Definition}
\newtheorem{remark}[proposition]{Remark}
\newtheorem{remarks}[proposition]{Remarks}

\begin{document}
\title[The Boundary of the Gelfand--Tsetlin Graph]{The Boundary of the Gelfand--Tsetlin Graph:\\ New Proof of Borodin--Olshanski's Formula,\\ and its $q$-analogue}
\author{Leonid Petrov}
\address{Department of Mathematics, Northeastern University, 360 Huntington ave., Boston, MA 02115, USA}
\address{Dobrushin Mathematics Laboratory, Kharkevich Institute for Information Transmission Problems, Moscow, Russia}
\email{lenia.petrov@gmail.com}

\begin{abstract}
  In the recent paper \cite{BorodinOlsh2011GT}, Borodin and Olshanski have presented a novel proof of the celebrated Edrei--Voiculescu theorem which describes the boundary of the Gelfand--Tsetlin graph as a region in an infinite-dimensional coordinate space. This graph encodes branching of irreducible characters of finite-dimensional unitary groups. Points of the boundary of the Gelfand--Tsetlin graph can be identified with finite indecomposable (= extreme) characters of the infinite-dimensional unitary group. An equivalent description identifies the boundary with the set of doubly infinite totally nonnegative sequences. 

  A principal ingredient of Borodin--Olshanski's proof is a new explicit determinantal formula for the number of semi-standard Young tableaux of a given skew shape (or of Gelfand--Tsetlin schemes of trapezoidal shape). We present a simpler and more direct derivation of that formula using the Cauchy--Binet summation involving the inverse Vandermonde matrix. We also obtain a $q$-generalization of that formula, namely, a new explicit determinantal formula for arbitrary $q$-specializations of skew Schur polynomials. Its particular case is related to the $q$-Gelfand--Tsetlin graph and $q$-Toeplitz matrices introduced and studied by Gorin \cite{Gorin2010q}.
\end{abstract}


\maketitle

\section{Introduction} 
\label{sec:introduction}

We begin with describing (in combinatorial terms) main results of the present paper, namely, a formula for the number of Gelfand--Tsetlin schemes of trapezoidal shape, and its $q$-generalization. In \S \ref{sec:the_boundary_of_the_gelfand_tsetlin_graph} below we explain how our results are related to (and motivated by) the Edrei--Voiculescu theorem which describes the boundary of the Gelfand--Tsetlin graph.

\subsection{Signatures and Gelfand--Tsetlin schemes} 
\label{sub:signatures_and_gelfand_tsetlin_schemes}

A \emph{signature} of length $N$ is a nonincreasing $N$-tuple of integers $\nu=(\nu_1\ge \ldots\ge\nu_N)\in\Z^{N}$. Let $\GT_N$ denote the set of all signatures of length $N$ (by agreement, $\GT_0=\{\varnothing\}$). The set $\GT_N$ parametrizes irreducible representations of the unitary group $U(N)$ \cite{Weyl1946}, so in the literature signatures are sometimes called \emph{highest weights}. Branching of irreducible representations of  unitary groups leads to the notion of \emph{interlacing of signatures} $\mu\in\GT_{N-1}$, $\nu\in\GT_N$:
\begin{align}\label{interlace}
  \nu_1\ge\mu_1\ge\nu_2\ge \ldots\ge\nu_{N-1}\ge\mu_{N-1}\ge\nu_{N}.
\end{align}
We denote this interlacing relation between signatures by $\mu\prec\nu$. By agreement, $\varnothing\prec\la$ for all $\la\in\GT_1$.

\begin{figure}[htbp]
	\begin{center}
		\includegraphics[width=135pt]{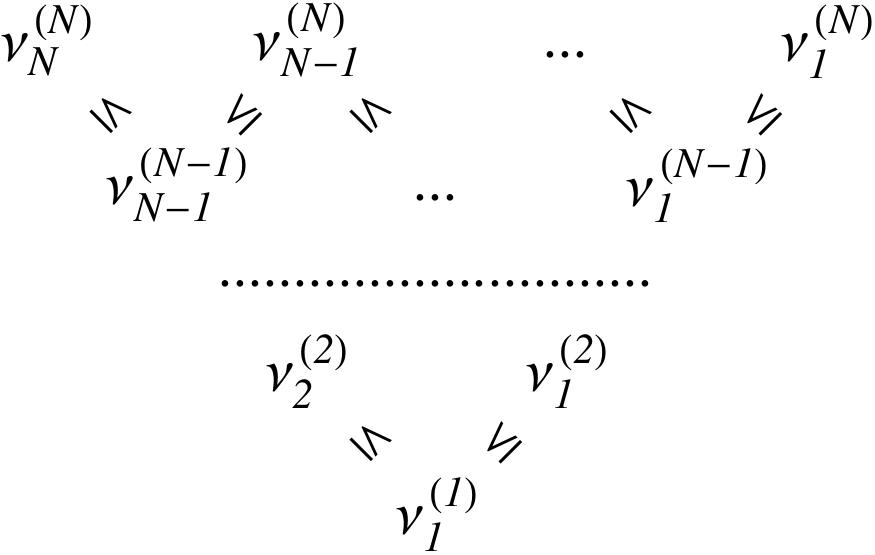}
	\end{center}
	\caption{A Gelfand--Tsetlin scheme of depth $N$.}
  \label{fig:GT_scheme}
\end{figure}

A \emph{Gelfand--Tsetlin scheme} of depth $N$ is a growing interlacing sequence of signatures:
\begin{align}\label{GT_scheme_sequence_of_signatures}
  \varnothing\prec\nu^{(1)}\prec\nu^{(2)}\prec
  \ldots\prec\nu^{(N-1)}\prec\nu^{(N)}.
\end{align}
One can also view this sequence as a triangular array of integers $\{\nu_{j}^{(m)}\}\in\Z^{N(N+1)/2}$ satisfying interlacing constraints indicated on Fig.~\ref{fig:GT_scheme}.


\subsection{Gelfand--Tsetlin graph} 
\label{sub:gelfand_tsetlin_graph}

The set of all signatures $\GT:=\bigsqcup_{N=0}^{\infty}\GT_N$ is traditionally equipped with a structure of a graded graph: we connect two signatures $\mu\in\GT_{N-1}$ and $\nu\in\GT_N$ by an edge iff $\mu\prec\nu$. This graded graph $\GT$ is called the \emph{Gelfand--Tsetlin graph}. 

A Gelfand--Tsetlin scheme with top row $\nu\in\GT_N$ is readily identified with a path in $\GT$ from the initial vertex $\varnothing\in\GT_0$ to $\nu$. Let $\Dim_N\nu$ denote the total number of such paths. In fact \cite{Weyl1946}, this number can be identified with the dimension of the irreducible representation of $U(N)$ indexed by the signature $\nu$. 

Also, define the \emph{relative dimension} $\Dim_{K,N}(\ka,\nu)$ to be the number of paths in the Gelfand--Tsetlin graph from the vertex $\ka\in\GT_K$ to the vertex $\nu\in\GT_N$, $K<N$. Such paths are identified with Gelfand--Tsetlin schemes of trapezoidal shape of depth $N-K+1$ with top row $\nu$ and bottom row $\ka$. If all the parts of $\ka$ and $\nu$ are nonnegative, $\Dim_{K,N}(\ka,\nu)$ also has an interpretation as the number of semi-standard Young tableaux of skew shape $\nu/\ka$ filled with numbers $K+1,\ldots,N$ \cite[Ch. I]{Macdonald1995}.


\subsection{Projections in Gelfand--Tsetlin schemes} 
\label{sub:projections_in_gelfand_tsetlin_schemes}

Consider the uniform probability measure $\Ps^{N,\nu}$ on the set of all Gelfand--Tsetlin schemes (\ref{GT_scheme_sequence_of_signatures}) with fixed top row $\nu\in\GT_N$. Clearly, $\Ps^{N,\nu}$ is supported on integer points inside some polyhedral region (of finite volume) in $\R^{N(N-1)/2}$ (with coordinates $\nu^{(1)},\ldots,\nu^{(N-1)}$). Fix $K<N$ and consider the projection of $\Ps^{N,\nu}$ onto the $K$th row $(\nu^{(K)}_{1},\ldots,\nu^{(K)}_{K})$ of the Gelfand--Tsetlin scheme. In this way, we get some (generally speaking, non-uniform) probability measure on $\GT_K\subset\Z^{K}$. According to the notation of \cite{BorodinOlsh2011GT}, we denote the probability of a signature $\ka\in\GT_K$ under this projected measure by $\La^{N}_{K}(\nu,\ka)$. Following \cite{BorodinOlsh2011Bouquet}, we call the stochastic matrix (=~Markov transition kernel) $\La^{N}_{K}$ of dimensions $\GT_N\times\GT_K$ ($K<N$) the \emph{link} from $\GT_N$ to $\GT_K$. 

It is readily seen that
\begin{align}\label{link_La}
	\La^{N}_{K}(\nu,\ka)=\Dim_K\ka\cdot 
	\frac{\Dim_{K,N}(\ka,\nu)}{\Dim_N\nu}.
\end{align}

As explained below in \S \ref{sec:the_boundary_of_the_gelfand_tsetlin_graph}, the Edrei--Voiculescu theorem (our main motivation in the present paper) boils down to the following question about the asymptotic behavior of the links (equivalent formulations of that theorem are discussed in \S \ref{sec:the_boundary_of_the_gelfand_tsetlin_graph}):

\begin{question}\label{question}
	Describe all possible sequences of signatures $\nu(1),\nu(2),\ldots$, where $\nu(N)\in\GT_N$, such that for every fixed level $K$ and signature $\ka\in\GT_K$, the sequence $\{\La^{N}_{K}(\nu(N),\ka)\}_{N\ge1}$ has a limit as $N$ goes to infinity. Such sequences $\{\nu(N)\}$ are called \emph{regular}.
\end{question}
The signatures $\nu(i)$'s do not need to interlace. Note that $\La^{N}_{K}(\nu(N),\ka)$ is well-defined only for $N>K$, but since $K$ is fixed and $N\to\infty$, the above question is well-posed. 

A possible approach to answering Question \ref{question} would be to obtain an explicit expression for the quantities $\Dim_{K,N}(\ka,\nu(N))/\Dim_N(\nu(N))$ (the $N$-dependent part of (\ref{link_La})) adapted to the desired asymptotic regime. Such an expression was first presented in \cite[Prop.~6.2]{BorodinOlsh2011GT}. We obtain an equivalent form of that expression (Theorem~\ref{thm:q=1_main_formula} below). 


\subsection{Number of trapezoidal Gelfand--Tsetlin schemes} 
\label{sub:main_result}

For $\nu\in\GT_N$, denote
\begin{align}\label{H_star}
	H^*(z;\nu):=\prod\nolimits_{r=1}^{N}\frac{z+r}{z+r-\nu_r}.
\end{align}

\begin{theorem}\label{thm:q=1_main_formula}
	For any $1\le K <N$, $\ka\in\GT_K$, and $\nu\in\GT_N$, we have the following formula:
	\begin{align}\label{skew_dim_det_A}
		\frac{\Dim_{K,N}(\ka,\nu)}{\Dim_N\nu}
		=
		\det[A_{i}(\ka_j-j)]_{i,j=1}^{K},
	\end{align}
	where\footnote{Here and below $(y)_m:=y(y+1)\ldots(y+m-1)$, $m=1,2,\ldots$ (with $(y)_0:=1$) denotes the Pochhammer symbol.}
	\begin{align}\label{A_i}
		A_i(x)=A_i(x\mid K,N,\nu):=\frac{N-K}{2\pi\i}
		\oint_{\Cont(x)}
		\frac{(z-x+1)_{N-K-1}}{(z+i)_{N-K+1}}H^*(z;\nu)dz.
	\end{align}
	Here the positively (counter-clockwise) oriented simple contour $\Cont(x)$ encircles points $x,x+1,\ldots,\nu_1-1$, and not the possible poles $x-1,x-2,\ldots,\nu_N-N$ coming from~$H^*(z;\nu)$.
\end{theorem}
\begin{proposition}\label{prop:equivalence_to_BO}
	Formula (\ref{skew_dim_det_A}) is equivalent to \cite[Prop. 6.2]{BorodinOlsh2011GT}.
\end{proposition}
Theorem \ref{thm:q=1_main_formula} admits a rather simple and direct proof which uses the Cauchy--Binet summation involving the inverse Vandermonde matrix. We present proofs of Theorem \ref{thm:q=1_main_formula} and Proposition \ref{prop:equivalence_to_BO} in \S \ref{sec:proof_of_theorem_thm:q=1_main_formula}.

\begin{remarks}\label{rmk:q=1_main_formula}
	{\bf1.} Since the quantity $\Dim_N\nu$ is given by a simple product formula (\ref{Dim_N_product}), Theorem \ref{thm:q=1_main_formula} essentially provides an explicit formula for the number of Gelfand--Tsetlin schemes of a given trapezoidal shape.

	\smallskip
	\noindent
	{\bf2.} It is known that the uniform measure $\Ps^{N,\nu}$ viewed as a measure on interlacing particle arrays (see Fig.~\ref{fig:tiling_particles}) is a determinantal point process. In Theorem 5.1 in \cite{Petrov2012} the correlation kernel of this measure was expressed as a double contour integral. This implies the existence of a $K\times K$ determinantal formula for the left-hand side of (\ref{skew_dim_det_A}) with matrix elements expressed as double contour integrals. However, Theorem \ref{thm:q=1_main_formula} provides a simpler formula involving only single contour integrals.
\end{remarks}

The formula of Theorem \ref{thm:q=1_main_formula} provides a very useful tool to approach Question~\ref{question}. Indeed, in Question~\ref{question} the level $K$ and the signature $\ka\in\GT_K$ are fixed, and the limit transition involves taking large $N$ and varying $\nu(N)$. Since the determinant in (\ref{skew_dim_det_A}) is of fixed size $K\times K$, in order to understand the behavior of $\La^{N}_{K}(\nu(N),\ka)$, one can start by considering asymptotics of the individual matrix elements $A_i(x\mid K,N,\nu(N))$ (\ref{A_i}), where $i$ and $x$ are fixed. It turns out that every $A_i(x)$ has a nice asymptotic behavior, and in this way Question~\ref{question} may be resolved. We discuss an approach to Question \ref{question} using Theorem \ref{thm:q=1_main_formula} in more detail in \S \ref{sec:the_boundary_of_the_gelfand_tsetlin_graph} and \S \ref{sec:idea_of_proof_of_the_uniform_approximation_theorem}. This method (using a formula equivalent to (\ref{skew_dim_det_A})) was suggested and carried out in~\cite{BorodinOlsh2011GT}.


\subsection{$q$-generalization} 
\label{sub:_q_generalization}

There is a $q$-deformation of Theorem \ref{thm:q=1_main_formula} which replaces the uniform probability measure $\Ps^{N,\nu}$ (\S \ref{sub:projections_in_gelfand_tsetlin_schemes}) by its certain $q$-version. The most general result in this direction we obtain is formulated in \S \ref{sub:_q_specializations_of_skew_schur_polynomials} (Theorem \ref{thm:general_q}) in terms of $q$-specializations of skew Schur polynomials. 

Here let us formulate a particular case related to the $q$-deformation of the Gelfand--Tsetlin graph introduced in \cite{Gorin2010q} (we recall the definition of the $q$-deformed graph in \S \ref{sub:_q_gelfand_tsetlin_graph}). Let us interpret Gelfand--Tsetlin schemes as 3D stepped surfaces (see Fig.~\ref{fig:tiling_particles}), and set the weight of each scheme (\ref{GT_scheme_sequence_of_signatures}) proportional to $q^{\vol}$, where $\vol$ is the (suitably defined) volume under the corresponding stepped surface. Such measures on 3D stepped surfaces inside a finite shape were considered in, e.g., \cite{CohnKenyonPropp2000}, \cite{OkounkovKenyon2007Limit}, \cite{borodin-gr2009q}, \cite{Petrov2012}.

\begin{figure}[htbp]
  \begin{tabular}{cc}
    \includegraphics[width=320pt]{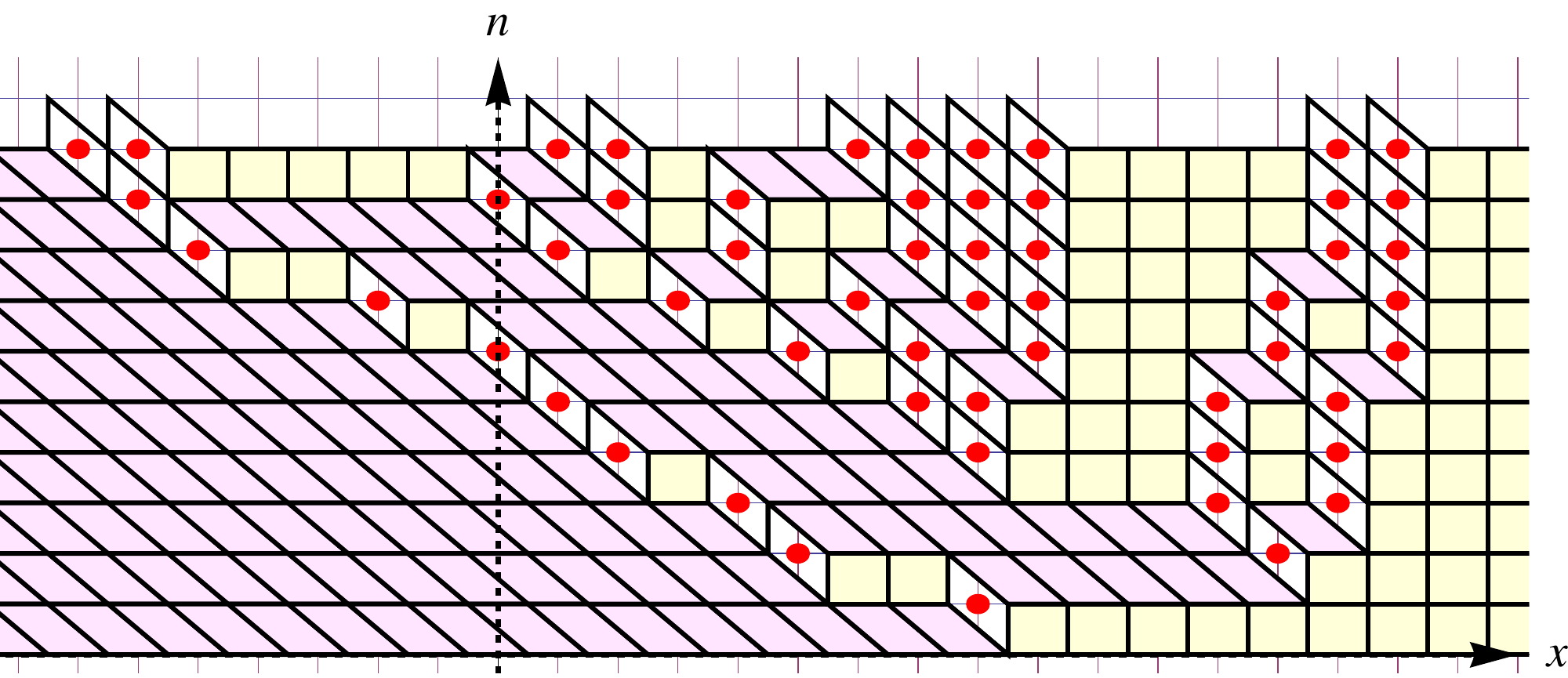}
  \end{tabular}
  \caption{Putting $n$ particles (red dots) $\x_{j}^{n}=\nu^{(n)}_j-j$, $j=1,\ldots,n$, on each $n$th horizontal line, $n=1,\ldots,N$, we obtain from the Gelfand--Tsetlin scheme (\ref{GT_scheme_sequence_of_signatures}) an interlacing particle configuration (\S \ref{sub:interlacing_arrays}) that can be interpreted as a lozenge tiling of the horizontal strip $0\le n\le N$ with $N$ small triangles added on top. This tiling may be also viewed as a 3D stepped surface.}
  \label{fig:tiling_particles}
\end{figure}

We will always assume that $0<q<1$. In the present paper we stick to the convention that the volume of a Gelfand--Tsetlin scheme (\ref{GT_scheme_sequence_of_signatures}) is equal to
\begin{align}\label{volume_GT_scheme}
	\vol(\nu^{(1)}\prec
  \ldots\prec\nu^{(N-1)}\prec\nu^{(N)}):=
  \sum_{n=1}^{N-1}|\nu^{(n)}|,\quad
  |\nu^{(n)}|:=\nu^{(n)}_{1}+\ldots+\nu^{(n)}_{n}.
\end{align}
Observe that we are summing over $N-1$ signatures because the $N$th signature $\nu^{(N)}=\nu$ is assumed to be fixed. Define \emph{the $q$-measure} $\q\Ps^{N,\nu}$ on the set of all Gelfand--Tsetlin schemes with fixed top row $\nu\in\GT_N$ by
\begin{align}\label{q_measure}
	\q\Ps^{N,\nu}
	(\nu^{(1)}\prec\ldots\prec\nu^{(N-1)}\prec\nu):=
	\frac1{\q \Dim_{N}\nu}\cdot{q^{\vol(\nu^{(1)}\prec\ldots\prec\nu^{(N-1)}\prec\nu)}},
\end{align}
where $\q\Dim_{N}\nu$ is a normalizing factor (partition function of the $q$-weighted triangular Gelfand--Tsetlin schemes). This is a $q$-analogue of $\Dim_N\nu$ defined in \S \ref{sub:gelfand_tsetlin_graph} (see also \S \ref{sub:number_of_triangular_gelfand_tsetlin_schemes}).

We define the \textit{$q$-links} $\q\La^{N}_{K}$ from $\GT_N$ to $\GT_K$ ($K<N$) using projections of the measure $\q\Ps^{N,\nu}$ in the same way as it was for $q=1$ in \S \ref{sub:projections_in_gelfand_tsetlin_schemes}. One can readily define the $q$-analogue $\q\Dim_{N,K}(\ka,\nu)$ of the number of trapezoidal Gelfand--Tsetlin schemes (see \S \ref{sub:proof_of_theorem_thm:q_main_formula}), so that 
\begin{align}\label{qlink_La}
	\q\La^{N}_{K}(\nu,\ka)=\q\Dim_K\ka\cdot 
	\frac{\q\Dim_{K,N}(\ka,\nu)}{\q\Dim_N\nu}.
\end{align}
\begin{theorem}\label{thm:q_main_formula}
	For any $1\le K <N$, $\ka\in\GT_K$, and $\nu\in\GT_N$, we have the following formula:
	\begin{align}\label{skew_dim_det_qA}
		\frac{\q\Dim_{K,N}(\ka,\nu)}{\q\Dim_N\nu}
		=
		(-1)^{K(N-K)}
		q^{(N-K)|\ka|}
		q^{-K(N-K)(N+2)/2}
		\det[\q A_{i}(\ka_j-j)]_{i,j=1}^{K},
	\end{align}
	where\footnote{Here and below $(a;q)_{m}:=(1-a)(1-aq)\ldots(1-aq^{m-1})$, $m=1,2,\ldots$ (with $(a;q)_{m}:=1$) denotes the $q$-Pochhammer symbol.}
	\begin{align}\label{qA_i}
		\q A_i(x)=\q A_i(x\mid K,N,\nu):=
		\frac{1-q^{N-K}}{2\pi\i}\oint_{\q\Cont(x)}dz
		\frac{(zq^{1-x};q)_{N-K-1}}{\prod_{r=i}^{N-K+i}(z-q^{-r})}
		\prod_{r=1}^{N}\frac{z-q^{-r}}{z-q^{\nu_r-r}}.
	\end{align}
	Here the positively (counter-clockwise) oriented simple contour $\q\Cont(x)$ encircles points $q^{x},q^{x+1},\ldots,q^{\nu_1-1}$, and not the possible poles $q^{x-1},q^{x-2},\ldots,q^{\nu_N-N}$.
\end{theorem}

\begin{proposition}\label{prop:q->1}
	In the $q\nearrow 1$ limit, Theorem \ref{thm:q_main_formula} becomes Theorem \ref{thm:q=1_main_formula}.
\end{proposition}

\begin{remark}[{cf. Remark \ref{rmk:q=1_main_formula}.2}]\label{rmk:q_main_formula}
	The determinantal kernel of the measure $\qi\Ps^{N,\nu}$ on Gelfand--Tsetlin schemes was computed in \cite[Thm. 4.1]{Petrov2012}. This result readily implies (by replacing $q$ by $q^{-1}$) the existence of a $K\times K$ determinantal formula as in (\ref{skew_dim_det_qA}) but with a much more complicated kernel expressed as a double contour integral containing a $q$-hypergeometric function ${}_2\phi_1$ inside. It seems remarkable that the technique of the present paper allows to obtain a much simpler single contour integral expression (\ref{qA_i}) for these matrix elements.
\end{remark}

We obtain Theorem \ref{thm:q_main_formula} in \S \ref{sec:_q_generalization_theorem_thm:q_main_formula_} as a corollary of a more general Theorem \ref{thm:general_q}. The latter deals with a larger class of $q$-measures on Gelfand--Tsetlin schemes than $\q\Ps^{N,\nu}$. We also explain how Theorem \ref{thm:q_main_formula} is related to the work of Gorin \cite{Gorin2010q} on the boundary of the $q$-Gelfand--Tsetlin graph and $q$-Toeplitz matrices. General $q$-measures on Gelfand--Tsetlin schemes whose projections (defined similarly to \eqref{qlink_La}; see Remark \ref{rmk:q_general_proj}) we compute in Theorem \ref{thm:general_q} allow to define other $q$-deformations of the classical Gelfand--Tsetlin graph. We plan to investigate boundaries of such deformations in a subsequent publication.


\subsection{Organization of the paper} 
\label{sub:organization_of_the_paper}

In \S \ref{sec:the_boundary_of_the_gelfand_tsetlin_graph} we briefly recall necessary definitions and results related to the Edrei--Voiculescu theorem describing the boundary of the Gelfand--Tsetlin graph. We also discuss various interpretations of these results in \S\S \ref{sub:representation_theoretic_meaning}--\ref{sub:totally_nonnegative_toeplitz_matrices}. The material of the latter two subsections is included to provide more background and motivations. In \S \ref{sub:uniform_approximation_theorem} we explain an approach to the Edrei--Voiculescu theorem employed in \cite{BorodinOlsh2011GT}. 

In \S \ref{sec:laurent_schur_polynomials} we recall the Laurent--Schur polynomials which provide a convenient algebraic framework for our proofs. 

In \S \ref{sec:proof_of_theorem_thm:q=1_main_formula} we prove Theorem \ref{thm:q=1_main_formula}, and show its equivalence to the Borodin--Olshanski's formula \cite[Prop.~6.2]{BorodinOlsh2011GT}. Then in \S \ref{sec:idea_of_proof_of_the_uniform_approximation_theorem} we briefly explain how the formula of Theorem \ref{thm:q=1_main_formula} leads to the Edrei--Voiculescu theorem.

In \S \ref{sec:_q_generalization_theorem_thm:q_main_formula_} we establish $q$-extensions of our results some of which are described in \S \ref{sub:_q_generalization}. In particular, we obtain Theorem \ref{thm:q_main_formula} as a corollary of a more general result (Theorem \ref{thm:general_q}) on $q$-specializations of skew Schur polynomials. 

\begin{remark}
	We have decided to present the proof of Theorem \ref{thm:q=1_main_formula} not as a $q\nearrow 1$ limit of Theorem \ref{thm:q_main_formula} (cf. Proposition \ref{prop:q->1}), but give instead a straightforward derivation in the  $q=1$ case which uses simpler notation than the argument for $0<q<1$. This allows to make the part of the paper about the ``classical'' ($q=1$) situation self-contained (e.g., in contrast with \cite{Petrov2012}). 
\end{remark}


\subsection{Acknowledgments} 
\label{sub:acknowledgments}

I am very grateful to Grigori Olshanski for drawing my attention to the problem, and to Alexei Borodin for useful comments. I would also like to thank Vadim Gorin for discussions regarding $q$-analogues. The work was partially supported by the RFBR-CNRS grants 10-01-93114 and 11-01-93105.



\section{The boundary of the Gelfand--Tsetlin graph} 
\label{sec:the_boundary_of_the_gelfand_tsetlin_graph}

\subsection{Coherent systems} 
\label{sub:coherent_systems}

There are several equivalent ways to define the boundary of a graded graph (such as the Gelfand--Tsetlin graph). Following, e.g., \cite{BorodinOlsh2011GT}, we use the notion of coherent systems.

\begin{definition}\label{def:coherent}
	Let $M_N$ be a probability measure on $\GT_N$ for each $N$. The sequence $\{M_N\}$ is called a \emph{coherent system} on $\GT$ if the measures $M_N$ are compatible with the links $\La^{N}_{N-1}$ (\S \ref{sub:projections_in_gelfand_tsetlin_schemes}):
	\begin{align*}
		M_N\La^{N}_{N-1}=M_{N-1},\qquad N\ge1, 
	\end{align*}
	or, in more detail (see (\ref{link_La}) and note that $\Dim_{N-1,N}(\mu,\la)=1$ if $\mu\prec\la$)
	\begin{align}\label{coherent_system_condition}
		\sum_{\nu\in\GT_N\colon\nu\succ\mu}
		M_N(\nu)
		\frac{\Dim_{N-1}\mu}{\Dim_{N}\nu}=M_{N-1}(\mu),
		\qquad \forall \mu\in\GT_{N-1}.
	\end{align}
	
	Coherent systems on $\GT$ form a convex set. A coherent system $\{M_N\}$ is called \emph{extreme} if it cannot be represented as a nontrivial convex combination $M_N=pM_N'+(1-p)M_N''$, $0<p<1$, of two other coherent systems $\{M_N'\}, \{M_N''\}$.
\end{definition}

\begin{definition}\label{def:boundary}
	The \emph{boundary} $\partial(\GT)$ of the Gelfand--Tsetlin graph $\GT$ is, by definition, the set of all extreme coherent systems on $\GT$.
\end{definition}

About connections of this notion with the minimal entrance boundary of a Markov chain, e.g., see \cite[\S2.2]{BorodinOlsh2011GT} and references therein.


\subsection{Connection to Question \ref{question}} 
\label{sub:connection_to_question_question}

Let us briefly discuss the Vershik--Kerov's idea (employed in, e.g., \cite{Vershik1975ergodic}, \cite{VK81AsymptoticTheory}, \cite{VK1981Characters}, \cite{VK82CharactersU}, \cite{vershik1987locally}) of approximating elements of the boundary $\partial(\GT)$ (= extreme coherent systems\footnote{Vershik and Kerov actually used an equivalent notion of extreme central measures on paths in the Gelfand--Tsetlin graph.}) by their finite-$N$ analogues.

Consider the part of the Gelfand--Tsetlin graph $\GT(N):=\GT_0\cup \ldots\cup \GT_N$ up to some fixed level $N$. Coherent systems on $\GT(N)$ are defined in the same way as for the whole graph $\GT$. Extreme coherent systems on $\GT(N)$ are in bijection with signatures $\nu\in\GT_N$. Namely, the extreme coherent system on $\GT(N)$ corresponding to $\nu\in\GT_N$ looks as
\begin{align}\label{M_K_on_GT(N)}
	M_K^{(N,\nu)}=\La^{N}_{K}\delta_{\nu}=\La^{N}_{K}(\nu,\cdot),\qquad
	K=0,1,\ldots,N-1,
\end{align}
where $\delta_{\nu}$ is the Dirac delta measure on $\GT_N$ supported at $\nu$, and $\La^{N}_{K}$ is the link~(\ref{link_La}).

As shown in \cite{Vershik1975ergodic} (see also \cite{VK81AsymptoticTheory}), every extreme coherent system on $\GT$ is a limit of those on $\GT(N)$ as $N\to\infty$. (The convergence of coherent systems is understood as (weak) convergence of their members, i.e., of the corresponding measures on $\GT_K$ for every fixed $K=1,2,\ldots$.) In detail, for every extreme coherent system $\{M_K\}_{K=0}^{\infty}$ on $\GT$ there exists a sequence of signatures $\nu(N)\in\GT_N$ ($\nu(N)$~is the index of the extreme coherent system on $\GT(N)$ for each $N$) such that for every fixed $K$ and $\ka\in\GT_K$ one has
\begin{align*}
	\lim_{N\to\infty}M_K^{(N,\nu(N))}(\ka)=M_K(\ka),
\end{align*}
where $M_K^{(N,\nu(N))}$ is defined in (\ref{M_K_on_GT(N)}). We see that in this way the problem of describing the boundary $\partial(\GT)$ becomes equivalent to Question \ref{question}. The sequence of signatures $\{\nu(N)\}$ above is the same as the regular sequence of Question~\ref{question}.


\subsection{Description of the boundary} 
\label{sub:description_of_the_boundary}

Let $\Om$ be a subset of the infinite-dimensional coordinate space $\R^{4\infty+2}$ defined by the following conditions:
\begin{align*}
	\Om:=\bigg\{\om=(\al^{+},\be^{+};\al^{-},\be^{-};\de^{+},\de^{-})\colon &
	\al^{\pm}_{1}\ge\al^{\pm}_{2}\ge \ldots\ge0;\;
	\be^{\pm}_{1}\ge\be^{\pm}_{2}\ge \ldots\ge0;
	\\&
	\sum\nolimits_{i=1}^{\infty}(\al^{\pm}_i+\be^{\pm}_i)
	\le\de^{\pm};
	\; \be^{+}_{1}+\be^{-}_{1}\le 1
	\bigg\}.
\end{align*}
As a subset of $\R^{4\infty+2}$ equipped with the product topology, the space $\Om$ is locally compact. Set $\ga^{\pm}:=\de^{\pm}-\sum_{i=1}^{\infty}(\al^{\pm}_i+\be^{\pm}_i)\ge0$. 

Let $\T:=\{u\in\C\colon|u|=1\}$ be the unit circle. Define the following function on $\T$ depending on $\om\in\Om$:
\begin{align}\label{Phi}
	\Phi(u;\om):=
	e^{\ga^+(u-1)+\ga^-(u^{-1}-1)}
	\prod_{i=1}^{\infty}
	\frac{1+\be_i^+(u-1)}{1-\al_i^+(u-1)}
	\frac{1+\be_i^-(u^{-1}-1)}{1-\al_i^-(u^{-1}-1)}.
\end{align}
Note that $\Phi(1;\om)=1$. Expand the function $\Phi(u;\om)$ as a Laurent series in $u$:
\begin{align}\label{Phi_Laurent}
	\Phi(u;\om)=\sum_{n=-\infty}^{\infty}\varphi_n(\om)u^{n}.
\end{align}
The Laurent coefficients $\varphi_n(\om)$ are themselves functions in $\om\in\Om$. They admit the following contour integral representations:
\begin{align}\label{phi_n_oint}
	\varphi_n(\om)=\frac{1}{2\pi\i}\oint_{\T}
	\Phi(u;\om)\frac{du}{u^{n+1}},\qquad n\in\Z.
\end{align}

Using the functions $\varphi_n(\om)$, define for every signature $\nu\in\GT_N$:
\begin{align}\label{varphi_nu}
	\varphi_\nu(\om):=\det[\varphi_{\nu_i-i+j}(\om)]_{i,j=1}^{N}.
\end{align}
Let the \emph{link from $\Om$ to $\GT_K$} for every $K$ be defined as
\begin{align}\label{link_infinity}
	\La^{\infty}_{K}(\om,\ka):=\Dim_K\ka\cdot \varphi_\ka(\om),\qquad
	\om\in\Om,\; \ka\in\GT_K.
\end{align}
It is not hard to show that (e.g., see \cite[Prop. 2.9]{BorodinOlsh2011GT}):
\begin{enumerate}[$\bullet$]
	\item $\La^{\infty}_{K}$ is indeed a link, i.e., 
	\begin{align*}
		\La^{\infty}_{K}(\om,\ka)\ge0,\qquad
		\sum\nolimits_{\ka\in\GT_K}\La^{\infty}_{K}(\om,\ka)=1.
	\end{align*}
	\item The links $\La^{N}_{K}$ and $\La^{\infty}_{K}$ are compatible in the sense that $\La^{\infty}_{N}\La^{N}_{K}=\La^{\infty}_{K}$ for $K<N$.
\end{enumerate}

\begin{figure}[htbp]
	\begin{center}
		\includegraphics[width=185pt]{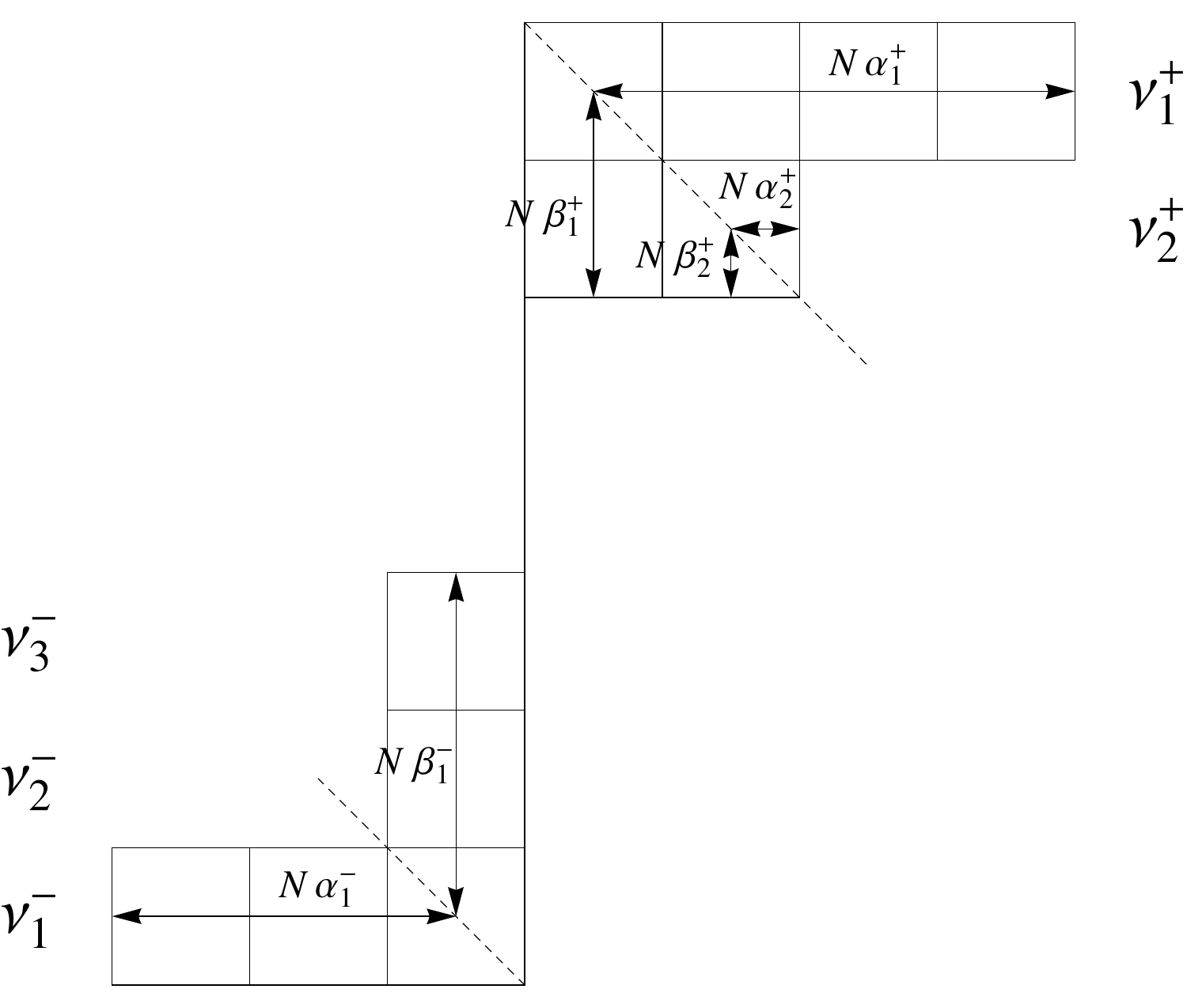}
	\end{center}
	\caption{Components $(\al^{\pm}(\nu),\be^{\pm}(\nu),\de^{\pm}(\nu))$ of the image $\om(\nu)$ of the signature $\nu\in\GT_N$ under the embedding $\GT_N\hookrightarrow\Om$. In this example $N=7$ and $\nu=(4,2,0,0,-1,-1,-3)$.}
  \label{fig:Frobenius}
\end{figure}

\begin{definition}\label{def:embeddings}
	Along with the links from $\Om$ to $\GT_N$, define embeddings $\GT_N\hookrightarrow \Om$, $\nu\mapsto\om(\nu)$ for $N=1,2,\ldots$ as follows. Write the signature $\nu$ as a union of ``positive'' and ``negative'' Young diagrams (= partitions, see \cite[Chapter I.1]{Macdonald1995})
	\begin{align*}
		\nu=(\nu_1^+,\ldots,\nu_{\ell^+}^+,-\nu_{\ell^-}^-,\ldots,-\nu_{1}^{-}),
		\qquad
		\nu_1^\pm\ge\ldots\ge\nu_{\ell^\pm}^\pm\ge0.
	\end{align*}
	Components $(\al^{\pm}(\nu),\be^{\pm}(\nu),\de^{\pm}(\nu))$ of the image $\om(\nu)$ look as (see Fig.~\ref{fig:Frobenius})
	\begin{align*}
		\al_i^{\pm}(\nu)=\frac{\nu^\pm_i-i+1/2}{N},\qquad
		\be_i^{\pm}(\nu)=\frac{(\nu^\pm)'_i-i+1/2}{N},\qquad
		\de^{\pm}(\nu)=\frac{|\nu^{\pm}|}{N},
	\end{align*}
	where $(\nu^\pm)'$ is the transposed Young diagram, and $|\nu^{\pm}|$ denotes the number of boxes in a diagram (in (\ref{volume_GT_scheme}), $|\nu|=|\nu^+|-|\nu^-|$). By agreement, if $\nu^\pm_i-i<0$, we set $\al_i^{\pm}(\nu)=0$, and similarly for $\be_i^{\pm}(\nu)$.
\end{definition}

\begin{theorem}[Edrei--Voiculescu]\label{thm:EV}
	There is a bijection between the boundary $\partial(\GT)$ of the Gelfand--Tsetlin graph and the space $\Om$ defined above. More precisely, 
	\begin{enumerate}[{\bf1.}]
		\item The extreme coherent system $\{M_N^{(\om)}\}$ corresponding to a point $\om\in\Om$ has the form (cf. (\ref{M_K_on_GT(N)}))
		\begin{align*}
			M_N^{(\om)}(\nu)=\La^{\infty}_{N}(\om,\nu),\quad 
			\mbox{for all $N=1,2,\ldots$ and $\nu\in\GT_N$}.
		\end{align*}
		\item For an extreme coherent system $\{M_N\}$, the corresponding point of $\Om$ is obtained as follows. Let $\om(M_N)$ denote the push-forward of the measure $M_N$ under the embedding $\GT_N\hookrightarrow\Om$. Then the measures $\om(M_N)$ on $\Om$ weakly converge to the delta measure which is supported at the point of $\Om$ corresponding to $\{M_N\}$.
	\end{enumerate}
\end{theorem}

In this form the Edrei--Voiculescu theorem was established in \cite{VK82CharactersU} (this is a note without proofs) and \cite{OkOl1998}. The latter paper's proof is based on the Binomial Formula for Schur (more generally, Jack) symmetric polynomials. In \S\S \ref{sub:representation_theoretic_meaning}--\ref{sub:totally_nonnegative_toeplitz_matrices} we describe equivalent formulations of Theorem \ref{thm:EV}. In \S \ref{sub:uniform_approximation_theorem} we explain the novel direct approach of Borodin and Olshanski \cite{BorodinOlsh2011GT} to Theorem \ref{thm:EV}, and put our result of Theorem \ref{thm:q=1_main_formula} in that framework.

\begin{remark}\label{rmk:reg_seq_answer}
	In the setting of Question \ref{question}, Theorem \ref{thm:EV} means that regular sequences of signatures $\{\nu(N)\}$, $\nu(N)\in\GT_N$, are characterized by the property that all rows, all columns of the Young diagrams $\nu^{\pm}(N)$, as well as their numbers of boxes $|\nu^{\pm}(N)|$  (cf. Definition \ref{def:embeddings} and Fig.~\ref{fig:Frobenius}) grow at most linearly in $N$.
\end{remark}


\subsection{Representation-theoretic interpretation} 
\label{sub:representation_theoretic_meaning}

Consider the increasing chain of finite-dimensional unitary groups
\begin{align}\label{U_Ns}
	U(1)\subset U(2)\subset U(3)\subset \ldots,
\end{align}
where the inclusions are defined as
\begin{align}\label{U_N_inclusion}
	U(N-1)\ni U\mapsto \begin{bmatrix}
		U&0\\
		0&1
	\end{bmatrix}\in U(N).
\end{align}

Let $U(\infty)$ be the union of the $U(N)$'s (\ref{U_Ns}). We equip $U(\infty)$ with the inductive limit topology.\footnote{In fact, in this topology $U(\infty)$ is neither a compact, nor even a locally compact group.} Every element $U\in U(\infty)$ lies in some $U(N)$ for large enough $N$, and thus has eigenvalues $(u_1,\ldots,u_N,1,1,\ldots)$, $u_i\in\T$.
\begin{definition}
	A \emph{character} $\chi$ of $U(\infty)$ is a function $\chi\colon U(\infty)\to\C$ which is 
	\begin{enumerate}[$\bullet$]
		\item continuous in the topology of $U(\infty)$ (i.e., restriction of $\chi$ to every $U(N)$ is continuous);
		\item constant on conjugacy classes of $U(\infty)$;
		\item positive definite;
		\item normalized so that $\chi(e)=1$, where $e=\mathrm{diag}(1,1,\ldots)\in U(\infty)$ is the unity of the group.
	\end{enumerate}
	Characters form a convex set, and so \emph{extreme characters} can be defined as extreme points of that set similarly to Definition \ref{def:coherent}.
\end{definition}

Extreme characters can be defined for any topological group. They serve as a natural replacement for the notion of irreducible characters (the latter make sense for, e.g., compact groups such as the $U(N)$'s). Extreme characters of $U(\infty)$ correspond to its finite factor representations \cite{Voiculescu1976}. In that paper, Voiculescu presented a list of extreme characters of $U(\infty)$ which are indexed by points of the space $\Om$ (\S \ref{sub:description_of_the_boundary}), and partially established completeness of that list. The value of the extreme character $\chi^{(\om)}$ corresponding to $\om\in\Om$ at an element $U\in U(\infty)$ with eigenvalues $(u_1,\ldots,u_N,1,1,\ldots)$ is 
\begin{align}\label{chi_om_product_Phi}
	\chi^{(\om)}(U)=\Phi(u_1;\om)\cdot\ldots\cdot \Phi(u_N;\om)=
	\prod_{u\in\text{spectrum of $U$}}\Phi(u;\om),
\end{align}
where $\Phi(u,\om)$ is given in (\ref{Phi}). In \cite{VK82CharactersU} the problem of describing characters of $U(\infty)$ was connected to the combinatorial Question \ref{question} (see \S \ref{sub:connection_to_question_question}). A connection of Voiculescu's work with earlier results on totally nonnegative Toeplitz matrices (see~\S \ref{sub:totally_nonnegative_toeplitz_matrices}) was discovered in \cite{Boyer1983}.

It should be mentioned that for the infinite symmetric group the same problem of classification of characters was solved by Thoma \cite{Thoma1964}.

For the finite-dimensional unitary groups $U(N)$ themselves, the extreme characters are precisely their normalized irreducible characters
\begin{align}\label{extreme_Schur_tilde}
	\tilde s_\nu(u_1,\ldots,u_N):=
	\frac{s_\nu(u_1,\ldots,u_N)}{s_\nu(1,\ldots,1)},\qquad \nu\in\GT_N
\end{align}
(where $s_\nu$'s are the Laurent--Schur polynomials (\ref{schur_polynomial})). Here $u_1,\ldots,u_N$ are eigenvalues of a unitary matrix $U\in U(N)$. The denominator $s_\nu(1,\ldots,1)=\Dim_N\nu$ is the dimension of the irreducible representation. This number also has a combinatorial interpretation as the number of Gelfand--Tsetlin schemes with fixed top row $\nu$ (\S \ref{sub:gelfand_tsetlin_graph}). The extreme characters of $U(\infty)$ are ``$N=\infty$'' analogues of the normalized irreducible characters $\tilde s_\nu$ of the $U(N)$'s (cf. \S\ref{sub:connection_to_question_question}).

Let us explain the connection of characters of $U(\infty)$ with coherent systems on the Gelfand--Tsetlin graph (\S \ref{sub:coherent_systems}). Restricting any character $\chi$ of $U(\infty)$ to $U(N)\subset U(\infty)$, one gets a normalized (but not necessary irreducible even if $\chi$ was extreme) character of $U(N)$. Let us write it as a linear combination of the normalized irreducibles $\tilde s_\nu$ with some coefficients $M_N(\nu)$:
\begin{align}\label{chi_N_expansion}
	\chi|_{{}_{U(N)}}=\sum\nolimits_{\nu\in\GT_N}M_N(\nu)\tilde s_\nu.
\end{align}
The numbers $\{M_N(\nu)\}_{\nu\in\GT_N}$ are nonnegative and sum to one, so they define a probability measure on $\GT_N$. Moreover, from the branching rule for the Laurent--Schur polynomials (\S \ref{sub:branching}) it follows that the probability measures $M_N$ on $\GT_N$ form a coherent system on $\GT$. In this way, characters of $U(\infty)$ are in one-to-one correspondence with coherent systems on the Gelfand--Tsetlin graph. Extreme characters correspond via (\ref{chi_N_expansion}) to extreme coherent systems. Thus, the result of Voiculescu on extreme characters of $U(\infty)$ is reformulated as Theorem \ref{thm:EV}. This reformulation is due to Vershik and Kerov, see citations in \S \ref{sub:connection_to_question_question}.


\subsection{Totally nonnegative Toeplitz matrices} 
\label{sub:totally_nonnegative_toeplitz_matrices}

\begin{definition}\label{def:TNTM}
	Let $\{b_n\}_{n\in\Z}$ be a sequence of nonnegative numbers such that $\sum_{n}b_n=1$. Consider the (doubly infinite) Toeplitz matrix $B:=[b_{j-i}]_{i,j\in\Z}$. The Toeplitz matrix $B$ and the sequence $\{b_n\}$ are called \emph{totally nonnegative} if all the minors (i.e., determinants of submatrices) of any order of the matrix $B$ are nonnegative. 
\end{definition}
Totally nonnegative Toeplitz matrices were classified by Edrei \cite{Edrei53} (see also \cite{AESW51}, \cite{ASW52}). The answer is that they are indexed by points of the same infinite-dimensional space $\Om$ (\S \ref{sub:description_of_the_boundary}). The generating function of the sequence $\{b_n^{(\om)}\}$ corresponding to $\om\in\Om$ has the form
\begin{align*}
	\sum\nolimits_{n\in\Z}b_n^{(\om)}u^{n}=\Phi(u;\om),
\end{align*}
where $\Phi(u;\om)$ is defined in (\ref{Phi}). In the notation of \S \ref{sub:description_of_the_boundary}, $b_n^{(\om)}=\varphi_n(\om)$.

It is worth noting that the classification of triangular totally nonnegative Toeplitz matrices (i.e., with $b_{-n}=0$, $n=1,2,\ldots$ in Definition \ref{def:TNTM}) which is also due to Edrei \cite{Edrei1952} (see also \cite{AESW51}), is equivalent to Thoma's description \cite{Thoma1964} of extreme characters of the infinite symmetric group.

Let us now explain how doubly infinite totally nonnegative Toeplitz matrices are related to the boundary of the Gelfand--Tsetlin graph. First, we need a lemma:

\begin{lemma}\label{lemma:Toeplitz_schur_expansion}
	Let a function $B(u)$, $u\in\T$, be expressed as a Laurent series $B(u)=\sum_{n\in\Z}b_n u^n$. Then for $u_1,\ldots,u_N\in\T$, we have
	\begin{align}\label{product_as_sum_over_schur}
		B(u_1)\cdot\ldots\cdot B(u_N)=
		\sum\nolimits_{\nu\in\GT_N}\det[b_{\nu_i-i+j}]_{i,j=1}^{N}
		\cdot
		s_\nu(u_1,\ldots,u_N).
	\end{align}
\end{lemma}
\begin{proof}
	This is a straightforward computation, e.g., see \cite[Lemme~2]{Voiculescu1976}. 

	Let us make a comment that the product $B(u_1)\ldots B(u_N)$ is a Laurent series in $u_1,\ldots,u_N$ which is of course symmetric in these variables. On the other hand, the Laurent--Schur polynomials $s_\nu(u_1,\ldots,u_N)$, where $\nu$ ranges over $\GT_N$, form a linear basis in the space $\R[u_{1}^{\pm1},\ldots,u_N^{\pm 1}]^{\mathfrak{S}(N)}$ of symmetric Laurent polynomials (\S \ref{sub:definition}). Identity (\ref{product_as_sum_over_schur}) is the explicit form of a representation of $B(u_1)\ldots B(u_N)$ as an (infinite) linear combination of the $s_\nu(u_1,\ldots,u_N)$'s.
\end{proof}

\subsubsection{From Toeplitz matrices to coherent systems} 
\label{ssub:from_toeplitz_matrices_to_coherent_systems}

Clearly, $\det[b_{\nu_i-i+j}]_{i,j=1}^{N}$ is a minor of the Toeplitz matrix $B$ of Definition \ref{def:TNTM}. If $B$ is totally nonnegative, then all the determinants of this form are nonnegative. Moreover, using the branching of Laurent--Schur polynomials (\S \ref{sub:branching}), one can readily check that the numbers
\begin{align*}
	M_N(\nu)=\Dim_N\nu\cdot  \det[b_{\nu_i-i+j}]_{i,j=1}^{N},\qquad
	\nu\in\GT_N,
\end{align*}
satisfy (\ref{coherent_system_condition}). Thus, $\{M_N\}$ is a coherent system on the Gelfand--Tsetlin graph. Moreover, it can be shown (e.g., see \cite{StratilaVoiculescu1982} or the approach of \cite{Olshanski1989_GK_Symm_eng}) that coherent systems of this form\footnote{That is, whose generating functions $\sum_{\nu\in\GT_N}M_N(\nu)\tilde s_\nu(u_1,\ldots,u_N)=B(u_1)\ldots B(u_N)$ (see (\ref{extreme_Schur_tilde}) and Lemma \ref{lemma:Toeplitz_schur_expansion}) are multiplicative in $u_1,\ldots,u_N$.} are extreme. In this way Edrei's classification of totally nonnegative Toeplitz matrices leads to Theorem \ref{thm:EV}.


\subsubsection{From coherent systems and characters to Toeplitz matrices} 
\label{ssub:from_coherent_systems_and_characters_to_toeplitz_matrices}

If $\{M_N\}$ is an extreme coherent system on the Gelfand--Tsetlin graph corresponding to a point $\om\in\Om$ (\S \ref{sub:representation_theoretic_meaning}), then from (\ref{chi_om_product_Phi}), (\ref{chi_N_expansion}), and (\ref{product_as_sum_over_schur}) we have $M_N(\nu)=\Dim_N\nu\cdot \varphi_\nu(\om)$, where $\varphi_\nu(\om)$ is defined in (\ref{varphi_nu}). The fact that we start from a character of $U(\infty)$ implies (via (\ref{chi_N_expansion})) that all the minors of the form $\varphi_\nu(\om)=\det[\varphi_{\nu_i-i+j}(\om)]_{i,j=1}^{N}$ for any $N$ and any $\nu\in\GT_N$ are nonnegative. These minors do not exhaust all possible minors. However, their nonnegativity is enough to conclude that the Toeplitz matrix $[\phi_{j-i}(\om)]_{i,j\in\Z}$ is totally nonnegative (e.g., see \cite[Thm. 9]{FZTP2000} and references after that Theorem). Thus, Theorem \ref{thm:EV} implies the result of Edrei on totally nonnegative Toeplitz matrices.
 


\subsection{Uniform Approximation Theorem} 
\label{sub:uniform_approximation_theorem}

Nowadays, there exist three different proofs of Theorem \ref{thm:EV}.\footnote{It is worth mentioning that yet another new way of establishing this theorem will appear soon in \cite{GorinPanova2012}.} The original works of Edrei and Voiculescu used theory of functions of a complex variable. The approach of \cite{OkOl1998} (outlined in \cite{VK81AsymptoticTheory}, \cite{VK82CharactersU}) answered Question \ref{question} about asymptotic behavior of the links $\La^{N}_{K}$ (see also~\S \ref{sub:connection_to_question_question}) by considering their generating function (in certain sense) and using the Binomial Theorem for Schur polynomials. 

The third, novel approach of \cite{BorodinOlsh2011GT} is based on a direct explicit formula for the links $\La^{N}_{K}$ (\ref{link_La}) which is equivalent to our Theorem \ref{thm:q=1_main_formula}. Using such a formula, it is possible to establish the following (see \S \ref{sec:idea_of_proof_of_the_uniform_approximation_theorem} and also \cite[\S\S 7--8]{BorodinOlsh2011GT}):

\begin{theorem}[Uniform Approximation Theorem {\cite[\S3]{BorodinOlsh2011GT}}]\label{thm:UAT}
	The finite-$N$ links (\ref{link_La}) are uniformly close to their ``$N=\infty$'' analogues (\ref{link_infinity}) in the sense that for any fixed $K=1,2,\ldots$ and $\ka\in\GT_K$, we have
	\begin{align*}
		\lim_{N\to\infty}
		\sup_{\nu\in\GT_N}
		\left|
		\La^{N}_{K}(\nu,\ka)-
		\La^{\infty}_{K}(\om(\nu),\ka)
		\right|=0,
	\end{align*}
	where $\om(\nu)\in\Om$ is described in Definition \ref{def:embeddings}.
\end{theorem}
As explained in \cite[\S3]{BorodinOlsh2011GT}, using some properties of the space $\Om$ and of the functions $\varphi_n(\om)$ (\ref{phi_n_oint}) on it, it is possible to deduce Theorem \ref{thm:EV} from Theorem~\ref{thm:UAT}. In fact, Theorem \ref{thm:UAT} also naturally implies an equivalent claim that the set of all (not necessary extreme) coherent systems on the Gelfand--Tsetlin graph is in a bijection with the space of Borel probability measures on $\Om$, see \cite[\S2.8]{BorodinOlsh2011GT}.

In \S \ref{sec:idea_of_proof_of_the_uniform_approximation_theorem} we will outline a proof of Theorem \ref{thm:UAT} using our result of Theorem \ref{thm:q=1_main_formula}.



\section{Laurent--Schur polynomials} 
\label{sec:laurent_schur_polynomials}

In this section we collect various definitions and results related to the Laurent--Schur polynomials. Most of them can be found in one form or another in \cite[Ch. I]{Macdonald1995}. Laurent--Schur polynomials provide a convenient framework for our proofs of Theorems \ref{thm:q=1_main_formula} and \ref{thm:q_main_formula}. Although the proof of Theorem \ref{thm:q=1_main_formula} can be given without any reference to Laurent--Schur polynomials, we use them in our proof in \S \ref{sec:proof_of_theorem_thm:q=1_main_formula} so that the ``classical'' ($q=1$) situation can be compared with its $q$-deformation discussed~in~\S \ref{sec:_q_generalization_theorem_thm:q_main_formula_}.

\subsection{Definition} 
\label{sub:definition}

Consider the space $\R[u_{1}^{\pm1},\ldots,u_N^{\pm 1}]^{\mathfrak{S}(N)}$ of symmetric Laurent polynomials in $N$ variables $u_1,\ldots,u_N$. A linear basis in this space is formed by the Laurent--Schur polynomials
\begin{align}\label{schur_polynomial}
  s_\nu(u_1,\ldots,u_N)=
  \frac{\det[u_i^{\nu_j+N-j}]_{i,j=1}^{N}}{\det[u_i^{N-j}]_{i,j=1}^{N}},
  \qquad \nu\in\GT_N,
\end{align}
indexed by all signatures of length $N$. Each $s_\nu(u_1,\ldots,u_N)$ is a homogeneous symmetric polynomial in $u_1^{\pm1},\ldots,u_N^{\pm1}$ of degree $|\nu|=\nu_1+\ldots+\nu_N\in\Z$ (this number is not necessary nonnegative). In particular, $s_{\varnothing}\equiv 1$.

Note that the denominator in (\ref{schur_polynomial}) is simply the Vandermonde determinant
\begin{align}\label{Vandermonde}
  V(u_1,\ldots,u_N):=\det[u_i^{N-j}]_{i,j=1}^{N}
  =\prod\nolimits_{1\le i<j\le N}(u_i-u_j).
\end{align}

For nonnegative signatures $\nu$ (i.e., for which $\nu_1\ge \ldots\ge\nu_N\ge0$), the Laurent--Schur polynomials become the ordinary Schur polynomials (which are honest symmetric polynomials in the variables $u_1,\ldots,u_N$). The ordinary Schur polynomials possess the \emph{stability property}:
\begin{align*}
	s_\nu(u_1,\ldots,u_N,u_{N+1}=0)=s_\nu(u_1,\ldots,u_N).
\end{align*}
(We append nonnegative signatures by zeroes as $\nu=(\nu_1,\ldots,\nu_\ell,0,0,\ldots)$, where $\ell$ is the number of positive parts in $\nu$.)

The Laurent--Schur polynomials $s_\nu(u_1,\ldots,u_N)$, $\nu\in\GT_N$ (if $u_1,\ldots,u_N\in\T$ are viewed as the eigenvalues of a matrix $U\in U(N)$), are exactly the irreducible characters of the unitary group $U(N)$ \cite{Weyl1946}.


\subsection{Particular cases: $e$-- and $h$--polynomials} 
\label{sub:particular_cases_e_and_h_polynomials}

Particular cases of Schur polynomials indexed by nonnegative signatures are the elementary symmetric polynomials
\begin{align}\label{e_m_particular_case}
	e_m(u_1,\ldots,u_N):=s_{(\underbrace{1,\ldots,1}_m,0,\ldots,0)}(u_1,\ldots,u_N)
\end{align}
and the complete homogeneous symmetric polynomials
\begin{align}\label{h_m_particular_case}
	h_m(u_1,\ldots,u_N):=s_{(m,0,\ldots,0)}(u_1,\ldots,u_N).
\end{align}
By agreement, in $N$ variables we have
\begin{align*}
	e_0=1,&\qquad e_{-1}=e_{-2}=\ldots=0,\qquad e_{N+1}=e_{N+2}=\ldots=0;\\
	h_0=1,&\qquad h_{-1}=h_{-2}=\ldots=0.
\end{align*}
The generating functions of the $e_m$'s and the $h_m$'s are given by (we write them in these forms for later convenience)
\begin{align}
	\sum\nolimits_{i=0}^{N}w^{N-i}(-1)^{i}e_i(u_1,\ldots,u_N)&=
	\prod\nolimits_{r=1}^{N}(w-u_r);
	\label{e_f_genf}
	\\
	\label{h_f_genf}
	\sum\nolimits_{i=0}^{\infty}
	t^{i}
	h_i(u_1,\ldots,u_N)&=
	\prod\nolimits_{r=1}^{N}\frac{1}{1-u_rt}.
\end{align}


\subsection{Branching} 
\label{sub:branching}

Having an irreducible character $s_\nu(u_1,\ldots,u_N)$ of the unitary group $U(N)$, one can restrict it to the subgroup $U(N-1)$ (see \S \ref{sub:representation_theoretic_meaning}) and represent it as a linear combination of irreducible characters of $U(N-1)$. This leads to the following \emph{branching rule} \cite{Weyl1946}:
\begin{align}\label{branching_rule}
	s_\nu(u_1,\ldots,u_{N-1};u_N=1)=
	\sum\nolimits_{\mu\in\GT_{N-1}\colon\mu\prec\nu}
  s_{\mu}(u_1,\ldots,u_{N-1}),
\end{align}
where the sum is taken over all signatures $\mu$ which interlace with $\nu$ (\ref{interlace}). In fact, a more general formula takes place:
\begin{align*}
  s_\nu(u_1,\ldots,u_{N-1},u_N)=
  \sum\nolimits_{\mu\in\GT_{N-1}\colon\mu\prec\nu}
  s_{\mu}(u_1,\ldots,u_{N-1})u_N^{|\nu|-|\mu|}.
\end{align*}
Continuing expansion for $u_{N-1},u_{N-2},\ldots,u_{1}$, we arrive at the following \emph{combinatorial formula for the Schur polynomials}:
\begin{equation}\label{schur_combinatorial_formula}
  s_\nu(u_1,\ldots,u_N)=
  \sum_{\varnothing
  \prec\nu^{(1)}\prec \ldots\prec\nu^{(N)}=\nu}
  u_1^{|\nu^{(1)}|}u_2^{|\nu^{(2)}|-|\nu^{(1)}|}
  \ldots
  u_N^{|\nu^{(N)}|-|\nu^{(N-1)}|},
\end{equation}
where the sum is taken over all Gelfand-Tsetlin schemes with fixed top row $\nu^{(N)}=\nu\in\GT_N$.


\subsection{Number of triangular Gelfand--Tsetlin schemes} 
\label{sub:number_of_triangular_gelfand_tsetlin_schemes}

Combinatorial formula (\ref{schur_combinatorial_formula}) readily implies that the number $\Dim_N\nu$ of triangular Gelfand--Tsetlin schemes with top row $\nu\in\GT_N$ (\S\S \ref{sub:signatures_and_gelfand_tsetlin_schemes}--\ref{sub:gelfand_tsetlin_graph}) is equal to 
\begin{align}\label{Dim_s_nu_at_ones}
	\Dim_N\nu=s_\nu(\underbrace{1,\ldots,1}_{N}) 
	\quad\mbox{for all $N=1,2,\ldots$, and $\nu\in\GT_N$}.
\end{align}
This number can be computed via its $q$-deformation using (\ref{schur_polynomial}):
\begin{align*}
	s_\nu(1,q,\ldots,q^{N-1})=
	\frac{\det[q^{(i-1)(\nu_j+N-j)}]_{i,j=1}^{N}}
	{\det[q^{(i-1)(N-j)}]_{i,j=1}^{N}}
	=
	\frac{V(q^{\nu_N-N},\ldots,q^{\nu_1-1})}{V(q^{-1},\ldots,q^{-N})},
\end{align*}
where $V(\cdot)$ is the Vandermonde determinant (\ref{Vandermonde}). Taking the $q\nearrow1$ limit above, we arrive at the following product formula:
\begin{align}\label{Dim_N_product}
	\Dim_N\nu=
	\frac{V({\nu_1-1},\ldots,{\nu_N-N})}{V({-1},\ldots,{-N})}
	=\prod_{1\le i<j\le N}\frac{\nu_i-\nu_j+j-i}{j-i}.
\end{align}

\begin{remark}
	In fact, as follows from (\ref{schur_combinatorial_formula}), the normalizing constant in (\ref{q_measure}) is given by 
	\begin{align}\label{s_nu_q}
		\q\Dim_N\nu=s_\nu(1,q,\ldots,q^{N-1})
		=\prod_{1\le i<j\le N}\frac{q^{\nu_i-i}-q^{\nu_j-j}}{q^{-i}-q^{-j}}.
	\end{align}
\end{remark}


\subsection{Skew polynomials} 
\label{sub:skew_polynomials}

For two signatures $\ka\in\GT_K$ and $\nu\in\GT_N$, $K<N$, define the \emph{skew Laurent--Schur polynomial} by the following combinatorial formula:
\begin{align}&
	\label{skew_schur_combinatorial_formula}
	s_{\nu/\ka}(u_{K+1},\ldots,u_N)\\&\hspace{15pt}:=
  \sum_{\ka=\nu^{(K)}
  \prec\nu^{(K+1)}\prec \ldots\prec\nu^{(N)}=\nu}
  u_{K+1}^{|\nu^{(K+1)}|-|\nu^{(K)}|}u_{K+2}^{|\nu^{(K+2)}|-|\nu^{(K+1)}|}
  \ldots
  u_N^{|\nu^{(N)}|-|\nu^{(N-1)}|},
  \nonumber
\end{align}
where the sum is taken over all Gelfand--Tsetlin schemes (of depth $N-K+1$) of trapezoidal shape with fixed top row $\nu\in\GT_N$ and bottom row $\ka\in\GT_K$. The skew polynomial $s_{\nu/\ka}$ is also a homogeneous Laurent polynomial, it has degree $|\nu|-|\ka|$. When all the parts of the signatures $\nu$ and $\ka$ are nonnegative, this is an ordinary (not Laurent) polynomial in $u_{K+1},\ldots,u_N$. In particular, $s_{\nu/\varnothing}=s_\nu$.

There is an identity which readily follows from (\ref{schur_combinatorial_formula}) and (\ref{skew_schur_combinatorial_formula}):
\begin{align}\label{Macdonald_identity}
	s_\nu(u_1,\ldots,u_N)=
	\sum_{\ka\in\GT_K}
	s_\ka(u_1,\ldots,u_K)
	s_{\nu/\ka}(u_{K+1},\ldots,u_N), \quad\nu\in\GT_N,
\end{align}
where $K<N$ is arbitrary and fixed, and the sum is taken over all signatures $\ka\in\GT_K$. Observe that this sum is actually finite. 

Formula (\ref{skew_schur_combinatorial_formula}) readily implies that the relative dimension $\Dim_{K,N}(\ka,\nu)$ (= number of trapezoidal Gelfand--Tsetlin schemes, \S \ref{sub:gelfand_tsetlin_graph}) is given in terms of skew Schur polynomials as
\begin{align}\label{skew_Dim_s_nu_at_ones}
	\Dim_{K,N}(\ka,\nu)=
	s_{\nu/\ka}(\underbrace{1,\ldots,1}_{N-K}),\qquad
	\nu\in\GT_N,\quad
	\ka\in\GT_K,\quad
	K<N.
\end{align}
This is a generalization of formula (\ref{Dim_s_nu_at_ones}) above.


\subsection{Jacobi--Trudi identities} 
\label{sub:jacobi_trudi_identities}

It is known that every symmetric (ordinary, not Laurent) polynomial can be expressed as a polynomial in the elementary symmetric polynomials $e_m$. The following \emph{Jacobi--Trudi identity} provides an explicit expression of this sort for the skew Schur polynomial $s_{\nu/\ka}$ ($\nu\in\GT_N$, $\ka\in\GT_K$), if all the parts of the signatures $\nu$ and $\ka$ are nonnegative:
\begin{align*}
	s_{\nu/\ka}(u_{K+1},\ldots,u_N)=\det\big[
	e_{\nu_i'-\ka_j'+j-i}(u_{K+1},\ldots,u_N)
	\big]_{i,j=1}^{\ell},
\end{align*}
where $\ka'$ and $\nu'$ are the transposed Young diagrams, and $\ell$ is any sufficiently large number. By agreement, we always append nonnegative signatures by zeroes: $\nu=(\nu_1,\ldots,\nu_N,0,0,\ldots)$, and $\ka=(\ka_1,\ldots,\ka_K,0,0,\ldots)$.

There is also a dual identity (again for nonnegative signatures $\ka$ and $\nu$):
\begin{align}\label{dual_Jacobi_Trudi}
	s_{\nu/\ka}(u_{K+1},\ldots,u_N)=\det\big[
	h_{\nu_i-\ka_j+j-i}(u_{K+1},\ldots,u_N)
	\big]_{i,j=1}^{N},
\end{align}
which has advantages that it involves the signatures $\ka$ and $\nu$ themselves, and not their transpositions, and also that the size of the determinant can be taken equal to $N$ (one can take any larger size as well).

A special case of (\ref{dual_Jacobi_Trudi}) is a formula for the ordinary Schur polynomials
\begin{align}\label{nonskew_dual_Jacobi_Trudi}
	s_{\nu}(u_{1},\ldots,u_N)=\det\big[
	h_{\nu_i+j-i}(u_{1},\ldots,u_N)
	\big]_{i,j=1}^{N},\qquad
	\nu\in\GT_N,\quad \nu_N\ge0.
\end{align}


\subsection{Interlacing arrays} 
\label{sub:interlacing_arrays}

When working with signatures and Gelfand--Tsetlin sche\-mes, it is sometimes useful to introduce shifted coordinates and regard every signature as a configuration of distinct particles on the integer lattice. Namely, let $\nu^{(1)}\prec \ldots\prec \nu^{(N)}$ be a Gelfand--Tsetlin scheme. Set
\begin{align*}
	\x^{m}_{j}:=\nu^{(m)}_j-j,\qquad
	m=1,\ldots,N,\quad j=1,\ldots,m.
\end{align*}
The array $\{\x^{m}_{j}\}$ satisfies the \emph{interlacing constraints}
\begin{align}\label{interlacing_particles}
	\x_{j+1}^{m}<\x_{j}^{m-1}\le \x_{j}^{m}
\end{align}
(for all $j$'s and $m$'s for which these inequalities can be written out), cf. \S \ref{sub:signatures_and_gelfand_tsetlin_schemes} and Fig.~\ref{fig:GT_scheme}. The configuration of particles at positions $(\x_{j}^{m},m)$ on the two-dimensional integer lattice can be interpreted as a lozenge tiling, see Fig.~\ref{fig:tiling_particles}.


\subsection{Skew polynomials in one variable} 
\label{sub:skew_polynomials_in_one_variable}

Let us consider the case when there is only one variable, say, $u$. Then we have (here and below $1_{\{\cdot\cdot\cdot\}}$ denotes the indicator of a set)
\begin{align}\label{h_one_variable}
	h_m(u)=1_{m\ge0}\cdot u^{m}, \qquad m\in\Z,
\end{align}
and for any two signatures $\nu\in\GT_N$, $\ka\in\GT_K$ (cf. \cite[(7)]{Borodin2010Schur}):
\begin{align}\label{skew_in_one_variable_power}
	s_{\nu/\ka}(u)=\begin{cases}
		u^{|\nu|-|\ka|},&\mbox{if $\ka\prec\nu$},\\
		0,&\mbox{otherwise}.
	\end{cases}
\end{align}
We see that $s_{\nu/\ka}(u)$ vanishes unless $K=N-1$ and the signature $\ka$ interlaces with $\nu$ as in (\ref{interlace}). 

We have the following determinantal formula (for \emph{nonnegative} signatures $\ka$ and~$\nu$; recall that we append them by zeroes) which follows from the Jacobi--Trudi identity (\ref{dual_Jacobi_Trudi}) and from (\ref{h_one_variable})
\begin{align}
	s_{\nu/\ka}(u)=\det\big[1_{\nu_i-i\ge\ka_j-j}\cdot u^{(\nu_i-i)-(\ka_j-j)}\big]_{i,j=1}^{N}.
	\label{skew_one_variable_positive}
\end{align}

Let us introduce a variant of this determinantal formula which works for arbitrary (not necessary nonnegative) signatures. It is based on the idea of \emph{virtual particles}, e.g., see \cite[\S4]{Borodin2009}, and also \cite{BorodinKuan2007U}. Similar idea was employed in \cite[\S4]{Petrov2012}. Let $\nu^{(N)}\in\GT_N$, and $\nu^{(N-1)}\in\GT_{N-1}$. Consider the particles $\x^{N}_{j},\x^{N-1}_j$ as in (\S \ref{sub:interlacing_arrays}). Let, by agreement, $\x^{N-1}_{N}=virt$ be a virtual particle. Denote 
\begin{align}\label{xi}
	\xi_u(x,y):=
	u^{y-x}1_{x\le y}+u^y1_{x=virt},\qquad
  x\in\{virt\}\cup\Z,\quad y\in\Z.
\end{align}
Informally, if $|u|<1$, one can think that $virt=-\infty$. Then 
\begin{align}\label{skew_one_variable_det_any}
	s_{\nu^{(N)}/\nu^{(N-1)}}(u)=u^{N}\cdot
	\det\big[\xi_{u}(\x_{i}^{N-1},\x_{j}^{N})\big]_{i,j=1}^{N}.
\end{align}
This determinantal formula follows from (\ref{skew_in_one_variable_power}). The additional factor $u^N$ comes form the fact that
\begin{align*}
	\sum\nolimits_{j=1}^{N}\x^{N}_{j}-
	\sum\nolimits_{i=1}^{N-1}\x^{N-1}_{i}=
	|\nu^{(N)}|-|\nu^{(N-1)}|-N.
\end{align*}

From (\ref{skew_schur_combinatorial_formula}) and (\ref{skew_one_variable_det_any}) it follows that the skew Schur polynomial $s_{\nu/\ka}$ for any $\nu$ and $\ka$ can be written as a sum of products of determinants. We use such a formula in our proof of Theorem \ref{thm:q=1_main_formula} in the next section, and also in \S \ref{sec:_q_generalization_theorem_thm:q_main_formula_}.



\section{Number of trapezoidal Gelfand--Tsetlin schemes\\(proof of Theorem \ref{thm:q=1_main_formula})} 
\label{sec:proof_of_theorem_thm:q=1_main_formula}

\subsection{Cauchy--Binet formula} 
\label{sub:cauchy_binet_formula}

Let us recall the well-known Cauchy--Binet formula in a form convenient for us. Let $A(x,y)$ and $B(y,z)$ be $\Z\times\Z$ matrices, and $AB$ is their product (assume that it is well-defined). Then for any two ordered $N$-tuples of integer indices $x_1>\ldots>x_N$ and $z_1>\ldots>z_N$, one has
\begin{align}\label{Cauchy_Binet}
	\det[(AB)(x_i,z_k)]_{i,k=1}^{N}=
	\sum_{y_1>\ldots>y_N}
	\det[A(x_i,y_j)]_{i,j=1}^{N}
	\det[B(y_j,z_k)]_{j,k=1}^{N},
\end{align}
where the sum is taken over all ordered $N$-tuples $y_1>\ldots>y_N$.


\subsection{Inverse Vandermonde matrix} 
\label{sub:inverse_vandermonde_matrix}

Here we discuss the inverse Vandermonde matrix --- an object that turns out to be very useful in our argument. This matrix was also used to obtain the determinantal kernel in \cite{Petrov2012}.

Let $a_1>\ldots>a_N$ be some points which we call \emph{nodes}, and consider the Vandermonde matrix $\Vs(a):=[a_i^{N-j}]_{i,j=1}^{N}$. Clearly,
\begin{align*}
	\det \Vs(a)=V(a_1,\ldots,a_N)=\prod\nolimits_{1\le i<j\le N}(a_i-a_j)\ne0.
\end{align*}
Let $\Vs(a)^{-1}$ be the inverse of the Vandermonde matrix.
\begin{proposition}\label{prop:inverse_V}
	The elements of the inverse Vandermonde matrix can be written as double contour integrals as follows:
	\begin{align*}
		[\Vs(a)^{-1}]_{ij}=\frac{1}{(2\pi\i)^{2}}
		\oint_{\cont(a_j)}dz
		\oint_{\cont(\infty)}
		\frac{dw}{w^{N+1-i}}
		\frac{1}{w-z}\prod_{r=1}^{N}\frac{w-a_r}{z-a_r},
		\qquad i,j=1,\ldots,N.
	\end{align*}
	Here $\cont(a_j)$ is any small positively oriented contour around $a_j$; the positively oriented contour $\cont(\infty)$ in $w$ contains $\cont(a_j)$ (without intersecting it) and is sufficiently large.
\end{proposition}
\begin{proof}
	Using the elementary symmetric polynomials $e_m$ (\S \ref{sub:particular_cases_e_and_h_polynomials}), one can write
  \begin{align}\label{V_inverse_e_i}
    [\Vs(a)^{-1}]_{ij}=(-1)^{i-1}
    \frac{e_{i-1}(a_1,\ldots,{a_{j-1}},a_{j+1},\ldots,a_N)}
    {\prod_{r\ne j}(a_j-a_r)}.
  \end{align}
  Indeed, this formula follows from the fact that every cofactor of the Vandermonde matrix $\Vs(a)$ can be identified with the numerator in the right-hand side of (\ref{schur_polynomial}) with $\nu$ of the form $(1^{m})=(1,\ldots,1)$ ($m$~ones) for some $m$ (cf. (\ref{e_m_particular_case})).

  Using (\ref{e_f_genf}), we have
  \begin{align}\nonumber
  	[\Vs(a)^{-1}]_{ij}&=
  	\frac{(-1)^{i-1}e_{i-1}(a_1,\ldots,{a_{j-1}},a_{j+1},\ldots,a_N)}
  	{\prod_{r\ne j}(a_j-a_r)}\\&=
  	\frac{1}{2\pi\i}\oint_{\cont(\infty)}
  	\frac{dw}{w^{N+1-i}}\prod_{r\ne j}\frac{w-a_r}{a_j-a_r}.
  	\label{Inverse_V_Lagrange}
  \end{align}
  Then it is not hard to see that (\ref{Inverse_V_Lagrange}) is the same as the claim of the proposition.
\end{proof}
Products of the inverse Vandermonde matrix with certain column vectors are readily computed in a closed form:
\begin{proposition}[Summation formula]\label{prop:Vandermonde_summation}
	Let $f$ be any polynomial of degree not exceeding $N-1$. Then we have
	\begin{align*}
		\sum_{j=1}^{N}[\Vs(a)^{-1}]_{ij}f(a_j)=
		[w^{N-i}]f(w),\qquad i=1,\ldots,N,
	\end{align*}
	where $[w^{N-i}](\cdots)$ means the coefficient by $w^{N-i}$ in $f(w)$. In other words, 
	\begin{align*}
		[w^{N-i}]f(w)=\frac{1}{2\pi\i}\oint_{\cont(\infty)}
		\frac{f(w)}{w^{N+1-i}}dw.
	\end{align*}
\end{proposition}
\begin{proof}
	Using (\ref{Inverse_V_Lagrange}), we see that 
	\begin{align*}
		\sum_{j=1}^{N}[\Vs(a)^{-1}]_{ij}f(a_j)=
		\frac{1}{2\pi\i}\oint_{\cont(\infty)}
		\frac{dw}{w^{N+1-i}}
		\sum_{j=1}^{N}f(a_j)\prod_{r\ne j}\frac{w-a_r}{a_j-a_r}.
	\end{align*}
	The sum over $j$ under the integral is the Lagrange interpolation polynomial of degree $N-1$ with $N$ nodes $a_1,\ldots,a_N$. Since $f$ is a polynomial of degree $\le N-1$, the interpolation is exact and the sum is simply equal to $f(w)$. This concludes the~proof.
\end{proof}


\subsection{First determinantal formula} 
\label{sub:first_determinantal_formula}

The goal of this subsection is to obtain a $K\times K$ determinantal formula for the quantity $\frac{\Dim_{K,N}(\ka,\nu)}{\Dim_N\nu}$ as in (\ref{skew_dim_det_A}) but first with a different (more complicated) kernel. Then in \S \ref{sub:linear_transform} we explain how to transform that formula into the desired claim of Theorem \ref{thm:q=1_main_formula}.

\begin{proposition}\label{prop:skew_dim_det_A_first}
	For any $1\le K <N$, $\ka\in\GT_K$, and $\nu\in\GT_N$, we have 
	\begin{align}\label{skew_dim_det_A_first}
		\frac{\Dim_{K,N}(\ka,\nu)}{\Dim_N\nu}
		=
		(N-1)!\ldots (N-K)!\cdot 
	\det[\tilde A_i(\ka_j-j)]_{i,j=1}^{K},
	\end{align}
	where
	\begin{align}
		\label{A_tilde}
		\tilde A_i(x)&=\tilde A_i(x\mid K,N,\nu)
		:=
		\frac{1}{(2\pi\i)^{2}}
		\oint_{\Cont(x)}dz
		\oint_{\cont(\infty)}
		\frac{dw}{w^{N+1-i}}\times\\&\hspace{110pt}\times
		\frac{(z-x+1)_{N-K-1}}{(N-K-1)!}
		\frac{1}{w-z}\prod_{r=1}^{N}\frac{w-\nu_r+r}{z-\nu_r+r}.
		\nonumber
	\end{align}
	Here the positively oriented contour $\Cont(x)$ in $z$ encircles points $x,x+1,\ldots,\nu_1-1$, and not $x-1,x-2,\ldots,\nu_N-N$ (this contour is the same as in Theorem \ref{thm:q=1_main_formula}). The positively oriented contour $\cont(\infty)$ in $w$ contains $\Cont(x)$ (without intersecting it) and is sufficiently large.
\end{proposition}
The rest of this subsection is devoted to proving Proposition \ref{prop:skew_dim_det_A_first}.

\smallskip
\noindent\textbf{Step 1.} Fix $N$ and $\nu\equiv\nu^{(N)}\in\GT_N$. Using (\ref{Dim_s_nu_at_ones}), we write
\begin{align*}
	\frac{1}{\Dim_N\nu}=V(-1,\ldots,-N)
	\cdot
	\det \Big[[\Vs(\nu_1-1,\ldots,\nu_N-N)^{-1}]_{ij}\Big]_{i,j=1}^{N},
\end{align*}
where $V(\cdot)$ is the Vandermonde determinant (\ref{Vandermonde}), and we also use the inverse Vandermonde matrix (\S \ref{sub:inverse_vandermonde_matrix}) with nodes $\nu_1-1>\ldots>\nu_N-N$.

Define the following functions on $\Z$ ($i=1,\ldots,N$):
\begin{align}\label{psi_q=1}
	\psi_i(x\mid N)=\psi_i(x\mid N,N,\nu):=
	\sum_{j=1}^{N}
	1_{x=\nu_j-j}\cdot[\Vs(\nu_1-1,\ldots,\nu_N-N)^{-1}]_{ij}.
\end{align}
An obvious but useful observation is that for any integers $y_1>\ldots>y_N$, we have
\begin{align}
	V(-1,\ldots,-N)\cdot\det[\psi_i(y_j\mid N)]_{i,j=1}^{N}=
  \frac{1_{y_1=\nu_1-1}\ldots 1_{y_N=\nu_N-N}}
  {\Dim_N\nu}.\label{one_over_dim_delta}
\end{align}

\smallskip
\noindent\textbf{Step 2.} Let us express the relative dimension through the skew Schur polynomial: $\Dim_{K,N}(\ka,\nu)=s_{\nu/\ka}(1,\ldots,1)$ ($N-K$ ones, see (\ref{skew_Dim_s_nu_at_ones})). Using (\ref{skew_schur_combinatorial_formula}) and (\ref{skew_one_variable_det_any}), we rewrite this quantity as 
\begin{align}\label{sum_of_determinants}
	s_{\nu/\ka}(1,\ldots,1)=
	\sum_{\x^{K+1},\ldots,\x^{N-1}}
	\det[\xi_{1}(\x_{i}^{K},\x_{j}^{K+1})]_{i,j=1}^{K+1}
	\ldots
	\det[\xi_{1}(\x_{i}^{N-1},\x_{j}^{N})]_{i,j=1}^{N}.
\end{align}
The sum is taken over all arrays of integers $\{\x^{m}_{j}\}_{m=K+1}^{N-1}$ of depth $N-K-1$; in the notation of (\ref{skew_schur_combinatorial_formula}), $\x^{m}_{j}=\nu^{(m)}_{j}-j$. The determinants ensure interlacing of the rows $\x^{m}$ as in (\ref{interlacing_particles}), cf. (\ref{skew_in_one_variable_power}) and (\ref{skew_one_variable_det_any}). The $K$th and $N$th rows are fixed, $\x^{K}_{j}=\ka_j-j$, and the same for $\x^{N}$ and $\nu$. The matrix elements $\xi_1(\cdot,\cdot)$ are given in (\ref{xi}) (with $u=1$). By agreement, we append every row of particles $\x^{1}_{1},\ldots,\x^{m}_{m}$ by the virtual particle $\x^{m}_{m+1}=virt$ as explained in \S \ref{sub:skew_polynomials_in_one_variable}.

\smallskip
\noindent\textbf{Step 3.} Now we can write our ratio of dimensions as the following sum:
\begin{align}&
	\label{pre_Cauchy_Binet}
	\frac{\Dim_{K,N}(\ka,\nu)}{\Dim_N\nu}
	=
	V(-1,\ldots,-N)
	\sum_{\x^{K+1},\ldots,\x^{N-1},\x^{N}}
	\det[\xi_{1}(\x_{i}^{K},\x_{j}^{K+1})]_{i,j=1}^{K+1}
	\times\\&\hspace{30pt}\times
	\det[\xi_{1}(\x_{i}^{K+1},\x_{j}^{K+2})]_{i,j=1}^{K+2}
	\ldots
	\det[\xi_{1}(\x_{i}^{N-1},\x_{j}^{N})]_{i,j=1}^{N}
	\det[\psi_i(\x_j^{N}\mid N)]_{i,j=1}^{N}.
	\nonumber
\end{align}
In the above sum, $\x^{K}$ is still fixed as in Step 2, but now we can also \emph{sum over~$\x^{N}$} because of (\ref{one_over_dim_delta}). The sum in the above form is adapted to performing the Cauchy--Binet summation (\S \ref{sub:cauchy_binet_formula}), see Step 5.

\smallskip
\noindent\textbf{Step 4.} Let us define convolutions in the usual way (e.g., see \cite[\S4]{Borodin2009}):
\begin{align*}
	(f*g)(x,z):=\sum_{y\in\Z}f(x,y)g(y,z),\qquad
	(g*h)(x):=\sum_{y\in\Z}g(x,y)h(y)
\end{align*}
for any functions $f(x,y)$, $g(x,y)$, and $h(x)$. 

Let for $K<N$, 
\begin{align}
	\psi_i(x\mid K)=\psi_i(x\mid K,N,\nu):=
	\big(\underbrace{\xi_1*\ldots *\xi_1}_{N-K}{}*\psi_i(\cdot\mid N)\big)(x).
	\label{psi_i_K}
\end{align}
\begin{lemma}[Vanishing property]\label{lemma:vanishing}
	For any $K<N$ and $x\le \nu_N-K-1$, one has
	\begin{align*}
		\psi_{i}(x\mid K,N,\nu)=
		\begin{cases}
			0,&\mbox{if $i=1,\ldots,K$};\\
			1/{(N-K-1)!},
			&\mbox{if $i=K+1$}.
		\end{cases}
	\end{align*}
	Moreover, $\psi_{i}(virt\mid K,N,\nu)$, $i=1,\ldots,K+1$, is given by the same formula.
\end{lemma}
Informally, one may think that $\psi_{i}(virt\mid K)=\lim_{x\to-\infty}\psi_{i}(x\mid K)$.
\begin{proof}
	First, observe that for $x,y\in\Z$ we have
	\begin{align*}
		\xi_1(x,y)=\frac{1}{2\pi\i}\oint_{|z|=1}\frac{dz}{z^{y-x+1}}\frac{1}{1-z}
	\end{align*}
	(cf. (\ref{h_f_genf}) and (\ref{h_one_variable})), so
	\begin{align*}
		(\xi_1^{*(N-K)})(x,y)=
		\frac{1}{2\pi\i}\oint_{|z|=1}\frac{dz}{z^{y-x+1}}\frac{1}{(1-z)^{N-K}}
		=1_{x\le y}\binom{N-K-1+y-x}{N-K-1}.
	\end{align*}
	This implies that
	\begin{align*}&
		(N-K-1)!\cdot \psi_i(x\mid K)
		=
		\sum\nolimits_{y\colon y\ge x}
		(y-x+1)_{N-K-1}\cdot \psi_i(y\mid N)
		\\&\hspace{45pt}=
		\sum\nolimits_{j=1}^{N}1_{\nu_j-j\ge x}\cdot
		(\nu_j-j-x+1)_{N-K-1}
		[\Vs(\nu_1-1,\ldots,\nu_N-N)^{-1}]_{ij}
	\end{align*}
	(we have used (\ref{psi_q=1})).

	Consider the following polynomial in $w$ of degree $N-K-1\le N-1$:
	\begin{align*}
		f(w):=(w-x+1)_{N-K-1}.
	\end{align*}
	It can be readily checked that for $x\le \nu_N-K-1$, one has
	\begin{align*}
		1_{\nu_j-j\ge x}\cdot f(\nu_j-j)=f(\nu_j-j),
		\qquad j=1,\ldots,N,
	\end{align*}
	due to the fact that $f(w)$ vanishes for $w=x-1,\ldots,x-N+K+1$. 

	Thus, one can apply Proposition \ref{prop:Vandermonde_summation} to the above sum over $N$ and obtain
	\begin{align*}
		(N-K-1)!\cdot \psi_i(x\mid K)=[w^{N-i}]f(w).
	\end{align*}
	This is zero for $i=1,\ldots,K$, and is equal to one for $i=K+1$. This establishes the ``non-virtual'' claim of the lemma.

	To prove the claim about $\psi_{i}(virt\mid K)$, observe that $\psi_{i}(y\mid K+1)=0$ for $i=1,\ldots,K+1$ and all sufficiently small $y\in\Z$ (this follows from the ``non-virtual'' claim). Since we have (see (\ref{xi}))
	\begin{align*}
		\psi_{i}(virt\mid K)=\sum\nolimits_{y\in\Z}\psi_{i}(y\mid K+1)
		\quad\mbox{and}\quad
		\psi_{i}(x\mid K)=\sum\nolimits_{y\colon y\ge x}\psi_{i}(y\mid K+1),
	\end{align*}
	we can add zero summands to the second sum over $y$, and conclude that $\psi_i(virt\mid K)=\psi_i(x\mid K)$ for $i=1,\ldots,K+1$ and all sufficiently small $x\in\Z$. This completes the proof.
\end{proof}

\smallskip
\noindent\textbf{Step 5.} Let us perform the Cauchy--Binet summation (\S \ref{sub:cauchy_binet_formula}) in (\ref{pre_Cauchy_Binet}). We do the summation first over $\x^{N}$, then over $\x^{N-1}$, etc., up to $\x^{K+1}$. The first summation gives
\begin{align*}
	\sum\nolimits_{\x^{N}}\det[\xi_{1}(\x_{k}^{N-1},\x_{j}^{N})]_{k,j=1}^{N}
	\det[\psi_i(\x_j^{N}\mid N)]_{i,j=1}^{N}
	=\det[\psi_i(\x_k^{N-1}\mid N-1)]_{i,k=1}^{N}.
\end{align*}
Using Lemma \ref{lemma:vanishing}, we see that the $N$th column of the matrix $[\psi_i(\x_k^{N-1}\mid N-1)]_{i,k=1}^{N}$ has zero entries except for the $(N,N)$-th element which is equal to $1$ (recall that $\x^{N-1}_{N}=virt$). This allows to replace the $N\times N$ determinant in the right-hand side above by the same determinant of size $N-1$. 

Continuing and summing over the row $\x^{m}$, $m=N-1,\ldots,K+1$, with the help of Lemma~\ref{lemma:vanishing} we will each time reduce the size of the determinant by one, and this will produce the factor $1/(N-m)!$. Thus, we have shown that
\begin{align}\label{skew_Dim_step5}
	\frac{\Dim_{K,N}(\ka,\nu)}{\Dim_N\nu}=
	(N-1)!\ldots (N-K)!\cdot 
	\det[\psi_i(\ka_j-j\mid K,N,\nu)]_{i,j=1}^{K}
\end{align}
because in (\ref{pre_Cauchy_Binet}) we also had a factor $V(-1,\ldots,-N)=0!1!\ldots (N-1)!$.

\smallskip
\noindent\textbf{Step 6.} Let us now explain how to write the quantities $\psi_i(x\mid K)$ entering the determinant in the right-hand side of (\ref{skew_Dim_step5}) as double contour integrals. We have (see the proof of Lemma \ref{lemma:vanishing})
\begin{align*}
	\psi_i(x\mid K)=
	\sum\nolimits_{j=1}^{N}1_{\nu_j-j\ge x}\cdot
	\frac{(\nu_j-j-x+1)_{N-K-1}}{(N-K-1)!}
	[\Vs(\nu_1-1,\ldots,\nu_N-N)^{-1}]_{ij}.
\end{align*}
By Proposition \ref{prop:inverse_V}, each $(i,j)$-th element $[\Vs(\nu_1-1,\ldots,\nu_N-N)^{-1}]_{ij}$ of the inverse Vandermonde matrix is written as a double contour integral; the $z$ contour there is around $\nu_j-j$, and the $w$ contour is any sufficiently large contour. Integrating over the $z$ contour amounts to picking up the residue at $z=\nu_j-j$. Thus, the above sum over $j$ such that $\nu_j-j\ge x$ can be rewritten as an integral over the $z$ contour encircling points $x,x+1,\ldots$, and not $x-1,x-2,\ldots$, with the quantity $\frac{(\nu_j-j-x+1)_{N-K-1}}{(N-K-1)!}$ replaced by $\frac{(z-x+1)_{N-K-1}}{(N-K-1)!}$.

In this way we get the double contour integral formula (\ref{A_tilde}) for the matrix elements in (\ref{skew_Dim_step5}). This argument completes the proof of Proposition \ref{prop:skew_dim_det_A_first}.


\subsection{Linear transformation and proof of Theorem \ref{thm:q=1_main_formula}} 
\label{sub:linear_transform}

We now aim to rewrite the $K\times K$ determinantal formula obtained in Proposition \ref{prop:skew_dim_det_A_first} and get the desired formula of Theorem \ref{thm:q=1_main_formula}. 

We claim that the $K\times K$ matrices in (\ref{skew_dim_det_A_first}) and (\ref{skew_dim_det_A}) are related by a rather simple row transformation. To see that, we perform the $w$ integration in (\ref{A_tilde}). Since we can choose our contours so that on them $|w|>|z|$, we may expand
\begin{align*}
	\frac{1}{w-z}=
	\sum\nolimits_{j\ge0}z^{j}w^{-j-1}.
\end{align*}
Then, since by (\ref{e_f_genf}),
\begin{align*}
	\frac{1}{2\pi\i}\oint_{\cont(\infty)}\frac{dw}{w^{N+1-i}}
	w^{-j-1}\prod_{r=1}^{N}(w-\nu_r+r)=
	(-1)^{i-j-1}e_{i-j-1}(\nu_1-1,\ldots,\nu_N-N),
\end{align*}
we have
\begin{align*}
	\tilde A_i(x)&=\frac{1}{2\pi\i}
	\oint_{\Cont(x)}dz
	\frac{(z-x+1)_{N-K-1}}{(N-K-1)!}
	\prod_{r=1}^{N}\frac{1}{z-\nu_r+r}\times
	\\&\hspace{100pt}\times
	\sum\nolimits_{j=0}^{i-1}z^{j}(-1)^{i-j-1}
	e_{i-j-1}(\nu_1-1,\ldots,\nu_N-N).
\end{align*}

Now observe that the index $i$ enters the expression for $\tilde A_i(x)$ only through the following polynomial in $z$:
\begin{align*}
	\tilde p_i(z):=\sum\nolimits_{j=0}^{i-1}z^{j}(-1)^{i-j-1}
	e_{i-j-1}(\nu_1-1,\ldots,\nu_N-N),\qquad
	i=1,\ldots,K.
\end{align*}
The polynomials $\tilde p_i$ are monic (i.e., with the leading term 1) and have degrees $0,1,\ldots,K-1$. Thus, applying a suitable row transformation to the $K\times K$ matrix in (\ref{skew_dim_det_A_first}), we may replace them with any other basis in the space $\R_{\le K-1}[z]$ of polynomials in $z$ of degree $\le K-1$, and this will affect only the constant factor in our $K\times K$ determinantal formula.

The quantities $A_i(x)$ in (\ref{skew_dim_det_A}) have the form
\begin{align*}
	A_i(x)&=\frac{1}{2\pi\i}
	\oint_{\Cont(x)}dz
	\frac{(z-x+1)_{N-K-1}}{(N-K-1)!}p_i(z)
	\prod_{r=1}^{N}\frac{1}{z-\nu_r+r},
\end{align*}
where
\begin{align*}
	p_i(z):=\frac{(N-K)!(z+1)_{N}}{(z+i)_{N-K+1}},\qquad i=1,\ldots,K.
\end{align*}
These polynomials all have degree $K-1$, and they clearly form a basis in $\R_{\le K-1}[z]$. To establish (\ref{skew_dim_det_A}), it remains to compute the determinant of the transition matrix from the basis $\{\tilde p_i\}$ to $\{p_i\}$.
\begin{lemma}\label{lemma:determinant_of_p_i_coeffs}
	For the matrix $T=[T_{ij}]_{i,j=1}^{K}$ such that $\sum_{i=1}^{K}\tilde p_i(z)T_{ij}=p_j(z)$, we have
	\begin{align*}
		\det T=(N-1)!\ldots (N-K)!.
	\end{align*}
\end{lemma}
\begin{proof}
	Since the matrix of coefficients of the polynomials $\{\tilde p_i(z)\}$ is unitriangular, it suffices to show that the determinant of the matrix of coefficients of $\{p_i(z)\}_{i=1}^{K}$ has determinant $(N-1)!\ldots (N-K)!$. That is, we need to show that
	\begin{align*}
		\det\Big[\frac{1}{2\pi\i}\oint_{|z|=1}
		\frac{p_i(z)}{z^{j}}dz\Big]_{i,j=1}^{K}=(N-1)!\ldots (N-K)!.
	\end{align*}
	It is not hard to see that the above determinant is equal to
	\begin{align*}
		\frac{1}{(2\pi\i)^{K}}
		\oint_{|z_1|=1}\ldots
		\oint_{|z_K|=1}
		\frac{dz_1 \ldots dz_K}{z_1(z_2)^{2}\ldots(z_K)^{K}}
		\det[p_i(z_j)]_{i,j=1}^{K}.
	\end{align*}
	Since the $p_i$'s form a basis in $\R_{\le K-1}[z]$, the determinant $\det[p_i(z_j)]_{i,j=1}^{K}$ must (up to a constant) coincide with the Vandermonde determinant $V(z_1,\ldots,z_K)$. This constant does not depend on $z_1,\ldots,z_K$ and can be computed as
	\begin{align*}
		\frac{\det[p_i(z_j)]_{i,j=1}^{K}}{V(z_1,\ldots,z_K)}=
		\frac{\det[p_i(-j)]_{i,j=1}^{K}}{V(-1,\ldots,-K)}
		=\frac{p_1(-1)\ldots p_K(-K)}{V(-1,\ldots,-K)}.
	\end{align*}
	In the last equality we used the fact that $p_i(-j)=0$ if $i>j$. We also have
	\begin{align*}
		p_i(-i)=(-1)^{i-1}(i-1)!(N-i)!,
	\end{align*}
	so
	\begin{align*}
		{\det[p_i(z_j)]_{i,j=1}^{K}}
		=
		(-1)^{K(K-1)/2}
		(N-1)!\ldots (N-K)!\cdot
		{V(z_1,\ldots,z_K)}.
	\end{align*}
	On the other hand, observe that
	\begin{align*}
    \frac1{(2\pi\i)^{K}}&
    \oint\limits_{|z_1|=1}\ldots\oint\limits_{|z_K|=1} 
    \frac{dz_1 \ldots dz_{K} }
    {z_1(z_2)^{2}\ldots (z_{K})^{K}}
    V(z_1,\ldots,z_K)\\&\quad=
    \mbox{coefficient by $z_2 (z_3)^2 \ldots (z_{K})^{K-1}$ in $V(z_1,\ldots,z_K)$}
    =(-1)^{K(K-1)/2}.
  \end{align*}
  This concludes the proof.
\end{proof}
With this lemma, Proposition \ref{prop:skew_dim_det_A_first} readily implies Theorem \ref{thm:q=1_main_formula}. 


\subsection{Comparison with \cite[Prop. 6.2]{BorodinOlsh2011GT}} 
\label{sub:comparison_with_cite_prop_6_2_borodinolsh2011gt_}

In this subsection we prove Proposition \ref{prop:equivalence_to_BO}. That is, we compare our formula for $\frac{\Dim_{K,N}(\ka,\nu)}{\Dim_N\nu}$ of Theorem \ref{thm:q=1_main_formula} with the formula obtained earlier by Borodin and Olshanski \cite[Prop. 6.2]{BorodinOlsh2011GT}. 

Let us recall the notation of \cite[Prop. 6.2]{BorodinOlsh2011GT}. Let $\Lb$ be a finite interval of integers. By $V_\Lb$ denote the space of rational functions in one variable $z$ which are regular everywhere including $z=\infty$, except that they may have simple poles at some points in $\Z\setminus\Lb$. This space is spanned \cite[Prop. 6.1]{BorodinOlsh2011GT} by the functions
\begin{align*}
	f_{\Lb,m}(z):=
	\frac{\prod_{x\in\Lb}(z-x)}{\prod_{x\in\Lb}(z-x-m)},
	\qquad m\in\Z.
\end{align*}
Every function $f$ from $V_{\Lb}$ can be expressed as a finite linear combination of $f_{\Lb,m}$'s, the coefficients of this expansion are denoted by $(f\colon f_{\Lb,m})$.

Recall the function $H^{*}(\cdot;\nu)$ (\ref{H_star}). As a rational function in $z$, for every $\nu\in\GT_N$ it lies in $V_{\Lb(N)}$, where $\Lb(N):=\{-N,\ldots,-1\}$. To formulate \cite[Prop.~6.2]{BorodinOlsh2011GT}, choose $\nu\in\GT_N$ and $\ka\in\GT_K$, $K< N$. For $j=1,\ldots,K$, denote 
\begin{align*}
	\Lb(N,j):=\{-N+K-j,\ldots,-j\}.
\end{align*}
The formula of \cite[Prop.~6.2]{BorodinOlsh2011GT} looks as
\begin{align}\label{BO_formula}
	\frac{\Dim_{K,N}(\ka,\nu)}{\Dim_N\nu}=
	\det\Big[
	\left(
	H^*(\cdot,\nu)\colon
	f_{\Lb(N,j),\ka_i-i+j}
	\right)
	\Big]_{i,j=1}^{K}.
\end{align}

In the rest of this subsection we show that our quantities $A_i$ (\ref{A_i}) are equal to
\begin{align}\label{A_i_and_BO}
	A_i(x)=A_i(x\mid K,N,\nu)=
	\left(
	H^*(\cdot,\nu)\colon
	f_{\Lb(N,i),x+i}
	\right),
\end{align}
for any $i=1,\ldots,K$ and $x\in\Z$. This will establish the equivalence of our formula (\ref{skew_dim_det_A}) with (\ref{BO_formula}) (and thus prove Proposition \ref{prop:equivalence_to_BO}).

Fix $i=1,\ldots,K$ and expand $H^*(z;\nu)$ into a finite linear combination:
\begin{align*}
	H^*(z;\nu)=\sum\nolimits_{p\in\Z}
	\left(
	H^*(\cdot,\nu)\colon
	f_{\Lb(N,i),x+p}
	\right)\cdot
	f_{\Lb(N,i),x+p}(z)
\end{align*}
Let us apply the integration of the form $g(z)\mapsto\frac{N-K}{2\pi\i}\oint_{\Cont(x)}\frac{(z-x+1)_{N-K-1}}{(z+i)_{N-K+1}}g(z)dz$ (see (\ref{A_i})) to the both sides of the above expansion. We see that to get (\ref{A_i_and_BO}), it suffices to show that
\begin{align}\label{delta_i_p}
	\frac{N-K}{2\pi\i}
	\oint_{\Cont(x)}
	\frac{(z-x+1)_{N-K-1}}{(z+i)_{N-K+1}}f_{\Lb(N,i),x+p}(z)dz=\delta_{i,p}.
\end{align}

It can be readily checked that the $f$'s above have the form
\begin{align*}
	f_{\Lb(N,i),x+p}(z)=\frac{(z+i)_{N-K+1}}{(z+i-x-p)_{N-K+1}}.
\end{align*}

Consider two cases:

\noindent
\textbf{($i=p$)} We have in this case
\begin{align*}&
	\frac{N-K}{2\pi\i}
	\oint_{\Cont(x)}
	\frac{(z-x+1)_{N-K-1}}{(z+i)_{N-K+1}}f_{\Lb(N,i),x+i}(z)dz
	\\&\hspace{130pt}=\frac{N-K}{2\pi\i}
	\oint_{\Cont(x)}
	\frac{1}{(z-x)(z-x+N-K)}dz.
\end{align*}
The only pole of the integrand inside $\Cont(x)$ is $z=x$, and the residue at this pole is equal to one.

\noindent
\textbf{($i\ne p$)} We have
\begin{align*}&
	\frac{N-K}{2\pi\i}
	\oint_{\Cont(x)}
	\frac{(z-x+1)_{N-K-1}}{(z+i)_{N-K+1}}f_{\Lb(N,i),x+p}(z)dz
	\\&\hspace{130pt}=
	\frac{N-K}{2\pi\i}
	\oint_{\Cont(x)}
	\frac{(z-x+1)_{N-K-1}}{(z+i-x-p)_{N-K+1}}dz.
\end{align*}
The zeroes of the numerator are 
\begin{align*}
	z=x-1,x-2,\ldots,x-N+K+1,
\end{align*}
and the zeroes of the denominator are
\begin{align*}
	z=x+p-i,x+p-i-1,\ldots,x+p-i-N+K.
\end{align*}
Observe that the integrand decays as $z^{-2}$ at $z=\infty$ and so has zero residue at infinity. Recall that the contour $\Cont(x)$ encircles points $x,x+1,x+2,\ldots$. It is readily seen that (1) if $p<i$, then the integrand has no poles inside $\Cont(x)$, (2) for $p>i$, all the poles are inside $\Cont(x)$. Thus, the integral vanishes in both cases. 

We thus have proven (\ref{delta_i_p}), and therefore established (\ref{A_i_and_BO}). This concludes the proof of Proposition \ref{prop:equivalence_to_BO}.



\section{Idea of proof of the Uniform Approximation Theorem} 
\label{sec:idea_of_proof_of_the_uniform_approximation_theorem}

As shown in \S\S7--8 of \cite{BorodinOlsh2011GT}, the determinantal formula (\ref{BO_formula}) for the relative dimensions $\frac{\Dim_{K,N}(\ka,\nu)}{\Dim_N\nu}$ implies the Uniform Approximation Theorem (Theorem~\ref{thm:UAT}). In this section for the sake of completeness we include an idea of proof of Theorem \ref{thm:UAT} based on our equivalent formula for the relative dimensions (\ref{skew_dim_det_A})--(\ref{A_i}). We omit certain technical details which are the same as in \cite[\S8]{BorodinOlsh2011GT}.

Let us first rewrite the quantities $A_i(x)$ (\ref{A_i}) as contour integrals over the unit circle $\T$:
\begin{proposition}\label{prop:contour_integral_representation}
	For any fixed $K$, $i$ and $x$, all $N>K+x+1$ and any $\nu\in\GT_N$ one has
	\begin{align}\label{A_i_circle_integral}
		A_i(x)=\frac{1}{2\pi\i}
		\oint_{\T}
		\Phi(u;\om(\nu))\cdot
		\frac{\big(\frac{N}{u-1}-x+\frac{1}{2}\big)_{N-K-1}}
		{\big(\frac{N}{u-1}+i-\frac{1}{2}\big)_{N-K+1}}
		\frac{N(N-K)u}{(u-1)^{2}}\frac{du}{u},
	\end{align}
	where $\om(\nu)$ and $\Phi(u;\om)$ are defined in \S \ref{sub:description_of_the_boundary}.
\end{proposition}
This statement is parallel to \cite[Prop. 8.1]{BorodinOlsh2011GT}, but seems somewhat simpler because it does not involve several different cases.
\begin{proof}
	The quantity $A_i(x)$ is given in (\ref{A_i}) by the single contour integral over the positively oriented contour $\Cont(x)$ which encircles points $x,x+1,\ldots$, and leaves outside $x-1,x-2,\ldots$. However, observe that all possible poles of the integrand
	\begin{align*}
		\frac{(z-x+1)_{N-K-1}}{(z+i)_{N-K+1}}H^*(z;\nu)
		=
		\frac{(z-x+1)_{N-K-1}}{(z+i)_{N-K+1}}\frac{(z+1)_{N}}{\prod_{r=1}^{N}(z+r-\nu_r)}
	\end{align*}
	belong to the set
	\begin{align*}
		\{\nu_1-1,\ldots,\nu_N-N\}\setminus\{x-1,x-2,\ldots,x-N+K+1\}.
	\end{align*}
	This readily implies that we can drag the contour $\Cont(x)$ to the left, and replace it by $\Cont(x-N+K+1)$ without changing the integral. 

	Note also that the integrand in (\ref{A_i}) has zero residue at $z=\infty$ because there it decays as $z^{-2}$. Thus, one can deform the contour $\Cont(x-N+K+1)$ so that it becomes the vertical line which crosses the real line to the left of $x-N+K+1$:
	\begin{align}\label{z_contour_line}
		z(t)=x-N+K+\tfrac12-\i t,\qquad -\infty<t<\infty.
	\end{align}

	We now perform a change of variable suggested in \cite[Prop. 5.2]{BorodinOlsh2011Bouquet}:
	\begin{align}\label{change_of_variables}
		z=-\frac{1}{2}+\frac{N}{u-1},\qquad 
		u=1+\frac{N}{z+\frac12}.
	\end{align}
	As shown in that proposition, we have
	\begin{align*}
		H^*(z;\nu)=\Phi(u;\om(\nu)).
	\end{align*}
	Clearly, $dz=-\dfrac{Nu}{(u-1)^2}\dfrac{du}{u}$. Thus, we obtain
	\begin{align}\label{A_i_cont_prime}
		A_i(x)=-\frac{1}{2\pi\i}
		\oint_{\Cont'(x-N+K+1)}
		\Phi(u;\om(\nu))\cdot
		\frac{\big(\frac{N}{u-1}-x+\frac{1}{2}\big)_{N-K-1}}
		{\big(\frac{N}{u-1}+i-\frac{1}{2}\big)_{N-K+1}}
		\frac{N(N-K)u}{(u-1)^{2}}\frac{du}{u},
	\end{align}
	which is almost the same as the desired claim (\ref{A_i_circle_integral}), except for the minus sign and the fact that the integral is over the contour $\Cont'(x-N+K+1)$ which is the image of (\ref{z_contour_line}) under our change of variables (\ref{change_of_variables}). That is, the contour in (\ref{A_i_cont_prime}) is 
	\begin{align}\label{u_contour_circle}
		u(t)=1+\frac{N}{x-N+K+1-\i t}, \qquad -\infty<t<\infty.
	\end{align}
	Take $N>x+K+1$. An elementary computation shows that this contour is a circle with center $1+\frac{N}{2(x-N+K+1)}$ and radius $\frac{N}{2|x-N+K+1|}$ passed in the negative (clockwise) direction.

	The integrand	in (\ref{A_i_cont_prime}) has a finite number of possible poles which arise from $\Phi(u;\om(\nu))$ (see \S \ref{sub:description_of_the_boundary} and especially Definition \ref{def:embeddings}):
	\begin{align*}
		u=1+\frac{1}{\al_i^+(\nu)}\in(1,\infty);\qquad
		u=1- \frac{1}{1+\al_i^-(\nu)}\in(-1,1),
	\end{align*}
	plus a pole at $u=1$ (corresponding to $z=\infty$ via (\ref{change_of_variables})) where the integrand has zero residue. Because $N>x+K+1$, the $u$ contour (\ref{u_contour_circle}) encircles all poles which are inside the unit circle and leaves outside the ones belonging to $(1,\infty)$. Thus, we can replace it by the unit circle $\T$ itself. The negative direction of the contour (\ref{u_contour_circle}) then eliminates the minus sign in (\ref{A_i_cont_prime}). This concludes the proof.
\end{proof}

It can be readily checked that in (\ref{A_i_circle_integral}) we have
\begin{align}\label{convergence_to_power_of_u}
	\Rs_{K,x,i}^{(N)}(u):=\frac{\big(\frac{N}{u-1}-x+\frac{1}{2}\big)_{N-K-1}}
	{\big(\frac{N}{u-1}+i-\frac{1}{2}\big)_{N-K+1}}
	\frac{N(N-K)u}{(u-1)^{2}}\to \frac{1}{u^{x+i}},\qquad N\to\infty,
\end{align}
uniformly in $u\in\T$ for fixed $K$, $x$, and $i$. Thus, every $A_i(x)$ has a nice asymptotic behavior. Namely, it is close to $\varphi_{i+x}(\om(\nu))$ (see (\ref{phi_n_oint}) and Definition \ref{def:embeddings}).

The rest of the proof of Theorem \ref{thm:UAT} is based on Proposition \ref{prop:contour_integral_representation} and on the above observation (\ref{convergence_to_power_of_u}). We need to show that $\La^{N}_{K}(\nu,\ka)$ (\ref{link_La}) is close to $\La^{\infty}_{K}(\om(\nu),\ka)$ (\ref{link_infinity}) for all fixed $K$ and $\ka\in\GT_K$, all large $N$ and any $\nu\in\GT_N$. Both links involve one and the same factor $\Dim_K\ka$, so we need to show that the following $K\times K$ determinants 
\begin{align*}
	\frac{\Dim_{K,N}(\ka,\nu)}{\Dim_N\nu}=\det[A_{i}(\ka_j-j)]_{i,j=1}^{K}\
	\quad\mbox{and}\quad
	\varphi_\ka(\om(\nu))=\det[\varphi_{\ka_j-j+i}(\om(\nu))]_{i,j=1}^{K}.
\end{align*}
are close to each other. Both these determinants admit similar $K$-fold contour integral representations with integration over the torus $\T^{K}:=\T\times \ldots\times \T$:
\begin{align*}&
	\det[A_{i}(\ka_j-j)]_{i,j=1}^{K}=\frac{1}{(2\pi\i)^{K}}
	\oint_{\T^{K}}\Phi(u_1;\om(\nu))\ldots \Phi(u_K;\om(\nu))
	\times\\&\hspace{170pt}\times
	\det[\Rs^{(N)}_{K,\ka_j-j,i}(u)]_{i,j=1}^{K}
	\frac{du_1}{u_1}\ldots \frac{du_K}{u_K},
	\\&
	\det[\varphi_{\ka_j-j+i}(\om(\nu))]_{i,j=1}^{K}
	=\frac{1}{(2\pi\i)^{K}}
	\oint_{\T^{K}}\Phi(u_1;\om(\nu))\ldots \Phi(u_K;\om(\nu))
	\times\\&\hspace{170pt}\times
	\det[u^{-(\ka_j-j+i)}]_{i,j=1}^{K}
	\frac{du_1}{u_1}\ldots \frac{du_K}{u_K}.
\end{align*}
Since (\ref{convergence_to_power_of_u}) implies that $\det[\Rs^{(N)}_{K,\ka_j-j,i}(u)]_{i,j=1}^{K}\to \det[u^{-(\ka_j-j+i)}]_{i,j=1}^{K}$ uniformly in $(u_1,\ldots,u_K)\in\T^K$, this implies the desired Uniform Approximation Theorem (Theorem \ref{thm:UAT}), and thus (as explained in \cite[\S3]{BorodinOlsh2011GT}) the description of the boundary of the Gelfand--Tsetlin graph.


\section{$q$-generalizations} 
\label{sec:_q_generalization_theorem_thm:q_main_formula_}

In this section we briefly discuss $q$-extensions of Theorem \ref{thm:q=1_main_formula}. We start with the most general statement, and then obtain Theorem \ref{thm:q_main_formula} as its corollary. We will also discuss in \S\S \ref{sub:_q_gelfand_tsetlin_graph}--\ref{sub:_q_toeplitz_matrices} some connections of Theorem \ref{thm:q_main_formula} with the $q$-Gelfand--Tsetlin graph and $q$-Toeplitz matrices of \cite{Gorin2010q}.

\smallskip

We will always assume that $0 < q < 1$.

\subsection{$q$-specializations of skew Schur polynomials} 
\label{sub:_q_specializations_of_skew_schur_polynomials}

In the language of Laurent--Schur polynomials, Theorem \ref{thm:q=1_main_formula} provides a $K\times K$ determinantal formula for 
\begin{align*}
	\frac{\Dim_{K,N}(\ka,\nu)}{\Dim_N\nu}=
	\frac{s_{\nu/\ka}(\overbrace{1,\ldots,1}^{N-K})}
	{s_{\nu}(\underbrace{1,\ldots,1}_{N})},\qquad
	\ka\in\GT_K,\quad\nu\in\GT_N,\quad 1\le K<N.
\end{align*}

Our $q$-generalization involves putting powers of $q$ instead of $1$'s in the numerator and in the denominator of the above formula. The ordinary (not skew) Laurent--Schur polynomial $s_\nu$ will always be evaluated at the geometric sequence $1,q,\ldots,q^{N-1}$. Using our approach with the inverse Vandermonde matrix, we manage to replace the $N-K$ ones in $s_{\nu/\ka}(1,\ldots,1)$ by \emph{any} subset of the geometric sequence $1,q,\ldots,q^{N-1}$, and there still exists some $K\times K$ determinantal formula for the quotient of $q$-specialized $s_{\nu/\ka}$ and $s_\nu$.

Let us introduce some notation. Let $F:=\{0,1,\ldots,N-1\}$, and $T=\{t_{1}<t_2<\ldots<t_{N-K}\}\subset F$ be any subset of size $N-K$. Define the following functions:
\begin{align}\label{q_psi}
	\q\psi^{T}_{i}(x)=
	\q\psi^{T}_{i}(x\mid K,N,\nu)
	:=
	\sum_{j=1}^{N}
	h_{\nu_j-j-x}(q^{T})
	\cdot[\Vs(q^{\nu_N-N},\ldots,q^{\nu_1-1})^{-1}]_{ij}
\end{align}
($i=1,\ldots,N$, $x\in\Z$), where $h_{m}(q^{T})$ is the complete homogeneous symmetric polynomial (\S \ref{sub:particular_cases_e_and_h_polynomials}) evaluated at $q^{t_1},\ldots,q^{t_{N-K}}$, and $\Vs(q^{\nu_N-N},\ldots,q^{\nu_1-1})^{-1}$ is the inverse Vandermonde matrix (\S \ref{sub:inverse_vandermonde_matrix}) with nodes $q^{\nu_N-N}>\ldots>q^{\nu_1-1}$.

Let $S:=F\setminus T$, and $S':=N-S$ (the operation is done with every element). Write $S'$ in increasing order, $S'=\{s_1'<\ldots<s_K'\}$.
\begin{theorem}\label{thm:general_q}
	With the above notation, we have the following $K\times K$ determinantal formula for any $1\le K<N$, $\ka\in\GT_K$, $\nu\in\GT_N$, and any subset $T\subset F$ of size $N-K$:\footnote{Note that the right-hand side of (\ref{q_general_skew_schur}) is clearly symmetric in $t_1,\ldots,t_{N-K}$, as it should be.}
	\begin{align}&
		\label{q_general_skew_schur}
		\frac{s_{\nu/\ka}(q^{t_1},\ldots,q^{t_{N-K}})}
		{s_\nu(1,q,q^{2},\ldots,q^{N-1})}
		\\&\hspace{40pt}=
		(-q^{N})^{t_1+\ldots+t_{N-K}}\cdot
		\frac{V(q^{-1},q^{-2},\ldots,q^{-N})}{V(q^{t_1},\ldots,q^{t_{N-K}})}
		\cdot\det[\q\psi^{T}_{s_i'}(\ka_j-j)]_{i,j=1}^{K},
		\nonumber
	\end{align}
	where $V(\cdot)$ is the Vandermonde determinant (\ref{Vandermonde}).
\end{theorem}

\begin{remark}\label{rmk:q_general_proj}
	One can define a measure on Gelfand--Tsetlin schemes with fixed top row $\nu\in\GT_N$ whose projections to every $K$th level, $K<N$, have the form
	\begin{align}\label{q_general_proj}
		s_\ka(q^{s_1},\ldots,q^{s_K})
		\frac{s_{\nu/\ka}(q^{t_1},\ldots,q^{t_{N-K}})}
		{s_\nu(1,q,q^{2},\ldots,q^{N-1})},\qquad \ka\in\GT_K
	\end{align}
	(cf. identity \eqref{Macdonald_identity}). For $T=\{0,1,\ldots,N-K-1\}$ we get the measure $\q\Ps^{N,\nu}$ (i.e.,~$q^{\vol}$), see \S \ref{sub:_q_generalization} and \S \ref{sub:proof_of_theorem_thm:q_main_formula} below. 

	Projections \eqref{q_general_proj} allow to define more general $q$-deformations of the Gelfand--Tsetlin graph than the one considered below in \S\S \ref{sub:_q_gelfand_tsetlin_graph}--\ref{sub:boundary_of_the_q_gelfand_tsetlin_graph}. We plan to discuss their boundaries (defined in the spirit of Question \ref{q_question}) in a subsequent publication.
\end{remark}

The rest of this subsection is devoted to proving Theorem \ref{thm:general_q}.

\smallskip

We argue as in \S \ref{sub:first_determinantal_formula}. Consider the functions
\begin{align*}
	\q\psi^{\varnothing}_{i}(x)=
	\q\psi^{\varnothing}_{i}(x\mid N,N,\nu)
	:=
	\sum\nolimits_{j=1}^{N}
	1_{x={\nu_j-j}}
	\cdot[\Vs(q^{\nu_N-N},\ldots,q^{\nu_1-1})^{-1}]_{ij}
\end{align*}
(which are particular cases of (\ref{q_psi})). For any integers $y_1>\ldots>y_N$ there is an obvious identity parallel to (\ref{one_over_dim_delta}) (see also (\ref{s_nu_q})):
\begin{align*}
	V(q^{-1},\ldots,q^{-N})\cdot
	\det[\q\psi^{\varnothing}_i(y_j\mid N,N,\nu)]_{i,j=1}^{N}=
  \frac{1_{y_1=\nu_1-1}\ldots 1_{y_N=\nu_N-N}}
  {s_\nu(1,q,\ldots,q^{N-1})}.
\end{align*}
Next, using (\ref{skew_schur_combinatorial_formula}) and (\ref{skew_one_variable_det_any}), we can write the skew Schur polynomial similarly to (\ref{sum_of_determinants}), which leads to the following expression (cf. (\ref{pre_Cauchy_Binet})):
\begin{align}
	\nonumber
	&
	\frac{s_{\nu/\ka}(q^{t_1},\ldots,q^{t_{N-K}})}
	{s_\nu(1,q,q^{2},\ldots,q^{N-1})}
	=V(q^{-1},\ldots,q^{-N})\cdot
	q^{Nt_{N-K}+(N-1)t_{N-K-1}+\ldots+(N-K+1)t_{1}}
	\times
	\\&
	\label{q_pre_Cauchy_Binet}
	\times
	\sum_{\x^{K+1},\ldots,\x^{N-1},\x^{N}}
	\det[\xi_{q^{t_1}}(\x_{i}^{K},\x_{j}^{K+1})]_{i,j=1}^{K+1}
	\times\\&\hspace{20pt}\times
	\det[\xi_{q^{t_2}}(\x_{i}^{K+1},\x_{j}^{K+2})]_{i,j=1}^{K+2}
	\ldots
	\det[\xi_{q^{t_{N-K}}}(\x_{i}^{N-1},\x_{j}^{N})]_{i,j=1}^{N}
	\det[\q\psi_i^{\varnothing}(\x_j^{N})]_{i,j=1}^{N}.
	\nonumber
\end{align}
This formula is adapted to performing the Cauchy--Binet summation (\S \ref{sub:cauchy_binet_formula}) as in Step 5 in \S \ref{sub:first_determinantal_formula}; but first we need to obtain an analogue of the vanishing property (Lemma \ref{lemma:vanishing}):
\begin{lemma}[$q$-vanishing property]
	\label{lemma:vanishing_property_q}
	For any subset $J=\{j_1,\ldots,j_\ell\}\subset F$, $\ell<N$, any $i=1,\ldots,N$, and any $x\le \nu_N-N+\ell-1$ we have
	\begin{align*}
		\q\psi^{J}_i(x)=\begin{cases}
			0,&\mbox{if $N-i\notin J$},\\
			q^{-x(N-i)}
			\prod_{r\in J,\, r\ne N-i}
			(1-q^{r-N+i})^{-1},&\mbox{otherwise}.
		\end{cases}
	\end{align*}
	Moreover,
	\begin{align*}
		\q\psi^{J}_i(virt)=\begin{cases}
			0,&\mbox{if $N-i\notin J$},\\
			\prod_{r\in J,\, r\ne N-i}
			(1-q^{r-N+i})^{-1},&\mbox{otherwise}.
		\end{cases}
	\end{align*}
\end{lemma}
Informally, one may think that $\q\psi^{J}_i(virt)=\lim_{x\to-\infty}\q\psi^{J}_i(x)$.
\begin{proof}
	First, observe that $h_m(q^{J})$, where $m=0,1,\ldots$, can be viewed as a polynomial in $q^{m}$. Indeed, by the very definition of the Schur polynomial (\ref{schur_polynomial}), we have
	\begin{align*}
		h_m(q^{j_1},\ldots,q^{j_\ell})
		=
		\frac{\det[q^{j_r(m\cdot 1_{s=1}+\ell-s)}]_{r,s=1}^{\ell}}
		{V(q^{j_1},\ldots,q^{j_{\ell}})},
	\end{align*}
	Expanding the determinant along the first column, we obtain
	\begin{align*}
		h_m(q^{j_1},\ldots,q^{j_\ell})
		&=
		\sum\nolimits_{k=1}^{\ell}
		(-1)^{k-1}
		(q^{m})^{j_k}
		q^{j_k(\ell-1)}
		\frac{V(q^{j_1},\ldots,q^{j_{k-1}},q^{j_{k+1}},\ldots,q^{j_{\ell}})}{V(q^{j_1},\ldots,q^{j_{\ell}})}
		\\&=
		\sum\nolimits_{k=1}^{\ell}
		(q^{m})^{j_k}
		q^{j_k(\ell-1)}
		\frac{1}{\prod_{r\ne k}(q^{j_k}-q^{j_r})}.
	\end{align*}
	This gives an explicit expression of $h_m(q^{J})$ as a polynomial $f(q^{m})$, where 
	\begin{align*}
		f(w):=
		\frac{\det[w^{j_r\cdot 1_{s=1}}q^{j_r(\ell-s)}]_{r,s=1}^{\ell}}
		{V(q^{j_1},\ldots,q^{j_{\ell}})}
		=\sum\nolimits_{k=1}^{\ell}
		w^{j_k}
		{\prod\nolimits_{r\ne k}(1-q^{j_r-j_k})^{-1}}.
	\end{align*}
 	Clearly, $\deg f=\max\{j\colon j\in J\}$ which is $\le N-1$, and this polynomial contains only powers $w^{j_1},\ldots,w^{j_k}$. Moreover, from the expression of $f(w)$ as a ratio of determinants it follows that
 	\begin{align*}
 		f(q^{-1})=\ldots=f(q^{-(\ell-1)})=0,
 	\end{align*}
 	because for these values of $w$ the determinant has two identical columns.

	Therefore, for any $x\le \nu_N-N+\ell-1$ by Proposition \ref{prop:Vandermonde_summation} we get
	\begin{align*}
		\q\psi^{J}_{i}(x)=
		\sum\nolimits_{j=1}^{N}
		f(q^{\nu_j-j-x})
		\cdot[\Vs(q^{\nu_N-N},\ldots,q^{\nu_1-1})^{-1}]_{ij}
		=
		[w^{N-i}]f(wq^{-x}).
	\end{align*}
	Thus, if $N-i\notin J$, this is zero, and otherwise we have
	\begin{align*}
		\q\psi^{J}_{i}(x)=
		q^{-x(N-i)}
		\prod\nolimits_{r\in J,\, r\ne N-i}
		(1-q^{r-N+i})^{-1}.
	\end{align*}
	It is not hard to check (similarly to the end of the proof of Lemma \ref{lemma:vanishing}) that $\q\psi^{J}_{i}(virt)$ is given by the limit of the above expression as $x\to-\infty$. This concludes the proof.
\end{proof}
Using Lemma \ref{lemma:vanishing_property_q}, we perform the Cauchy--Binet summation (similarly to Step 5 in \S \ref{sub:first_determinantal_formula}) in (\ref{q_pre_Cauchy_Binet}) first over $\x^{N}$, then over $\x^{N-1}$, etc., up to $\x^{K+1}$. Every such summation reduces the size of the determinant by one. For example, in the first summation we have 
\begin{align*}
	\sum\nolimits_{\x^{N}}
	\det[\xi_{q^{t_{N-K}}}(\x_{i}^{N-1},\x_{j}^{N})]_{i,j=1}^{N}
	\det[\q\psi_i^{\varnothing}(\x_j^{N})]_{i,j=1}^{N}
	=
	\det[\q\psi_i^{\{t_{N-K}\}}(\x_j^{N-1})]_{i,j=1}^{N}.
\end{align*}
The $N$th column of the matrix in the right-hand side (corresponding to $\x^{N-1}_{N}=virt$) has zero entries except for the $(N-t_{N-k})$th one which is equal to one by Lemma \ref{lemma:vanishing_property_q}. The same reduction happens after every summation, and each time we use Lemma \ref{lemma:vanishing_property_q}. It is not hard to see that the resulting factor which arises after these reductions, combined with what was already present in (\ref{q_pre_Cauchy_Binet}), gives the desired pre\-factor $(-q^{N})^{t_1+\ldots+t_{N-K}}\frac{V(q^{-1},q^{-2},\ldots,q^{-N})}{V(q^{t_1},\ldots,q^{t_{N-K}})}$ in (\ref{q_general_skew_schur}). 

Thus, we have established Theorem \ref{thm:general_q}.


\subsection{Remark: contour integral representation in Theorem \ref{thm:general_q}} 
\label{sub:remark_on_contour_integral_representations_in_theorem_thm:general_q}

Using Proposition \ref{prop:inverse_V} and the proof of Lemma \ref{lemma:vanishing_property_q}, one can suggest the following double contour integral representation for the functions $\q\psi_i^{T}(x\mid K,N,\nu)$ (\ref{q_psi}) entering Theorem \ref{thm:general_q}:
\begin{align*}
	\q\psi_i^{T}(x\mid K,N,\nu)&=
	\frac{1}{(2\pi\i)^{2}}\oint_{\q\Cont(x)}dz
	\oint_{\cont(\infty)}\frac{dw}{w^{N+1-i}}
	\frac{1}{w-z}\prod\nolimits_{r=1}^{N}\frac{w-q^{\nu_r-r}}{z-q^{\nu_r-r}}
	\times\\&\hspace{100pt}\times
	\sum\nolimits_{k=1}^{N-K}z^{t_k}q^{-xt_k}
	\prod\nolimits_{s\ne k}(1-q^{t_s-t_k})^{-1}.
\end{align*}
The contour $\q\Cont(x)$ is the same as in Theorem \ref{thm:q_main_formula}, and $\cont(\infty)$ is any sufficiently big contour containing $\q\Cont(x)$. 

We see that a $q=1$ statement parallel to Theorem \ref{thm:general_q} is Proposition \ref{prop:skew_dim_det_A_first} and not Theorem \ref{thm:q=1_main_formula}. In the general setting of Theorem \ref{thm:general_q} it is not clear whether it is possible to perform a linear transformation of rows in the $K\times K$ matrix in (\ref{q_general_skew_schur}) so that the new matrix elements would have simpler form (e.g., as it was done for $q=1$ in \S \ref{sub:linear_transform}). In the rest of this section we restrict our attention to the special case when the $q$-specialization $q^{t_{1}},\ldots,q^{t_{N-K}}$ in (\ref{q_general_skew_schur}) is a geometric sequence. This allows to perform the same trick as in \S \ref{sub:linear_transform}, and obtain Theorem \ref{thm:q_main_formula} in which the matrix elements admit a single contour integral representation. We discuss this in the next subsection.


\subsection{Proof of Theorem \ref{thm:q_main_formula}} 
\label{sub:proof_of_theorem_thm:q_main_formula}

Define the $q$-analogue of the number of trapezoidal Gelfand--Tsetlin schemes of depth $N-K+1$ with top row $\nu\in\GT_N$ and bottom row $\ka\in\GT_K$, $K<N$, by
\begin{align}\label{skew_dim_q}
	\q\Dim_{K,N}(\ka,\nu):=q^{|\ka|(N-K)}s_{\nu/\ka}(1,q,\ldots,q^{N-K-1}),
\end{align}
where $s_{\nu/\ka}$ is the skew Schur polynomial (\S \ref{sub:skew_polynomials}). By (\ref{skew_schur_combinatorial_formula}), one may say that $\q\Dim_{K,N}(\ka,\nu)$ is the partition function of trapezoidal Gelfand--Tsetlin schemes
\begin{align*}
	\ka\prec\nu^{(K+1)}\prec \ldots\prec\nu^{(N-1)}\prec\nu,
\end{align*}
where the weight of every particular scheme is proportional to
\begin{align*}
	q^{|\ka|}q^{|\nu^{(K+1)}|+|\nu^{(K+2)}|+\ldots+|\nu^{(N-1)}|}.
\end{align*}
The factor $q^{|\ka|}$ (which does not depend on a particular trapezoidal Gelfand--Tsetlin scheme) is introduced so that the $q$-link $\q\La^{N}_{K}$ (defined in \S \ref{sub:_q_generalization} as a projection of the $q$-measure (\ref{q_measure})) is given by formula (\ref{qlink_La}) which is similar to the corresponding $q=1$ formula (\ref{link_La}). 

In terms of Schur polynomials, Theorem \ref{thm:q_main_formula} gives a $K\times K$ determinantal formula for 
\begin{align}\label{last_section_q_relative_dim}
	\frac{\q\Dim_{K,N}(\ka,\nu)}{\q\Dim_N\nu}=q^{(N-K)|\ka|}
	\frac{s_{\nu/\ka}(1,q\ldots,q^{N-K-1})}{s_\nu(1,q,q^2,\ldots,q^{N-1})}.
\end{align}
In order to prove it, first observe that a particular case of Theorem \ref{thm:general_q} for $T=\{0,1,\ldots,N-K-1\}$ gives
\begin{align}&\label{last_section_first_det_formula}
	\frac{s_{\nu/\ka}(1,\ldots,q^{N-K-1})}{s_\nu(1,q,q^2,\ldots,q^{N-1})}
	=(-q^{N})^{(N-K)(N-K-1)/2}\times\\&
	\hspace{120pt}
	\times
	\frac{V(q^{-1},\ldots,q^{-N})}{V(1,q,\ldots,q^{N-K-1})}
	\cdot\det[\q\psi^{T}_{i}(\ka_j-j)]_{i,j=1}^{K}.
	\nonumber
\end{align}
We have for $m\ge0$:
\begin{align*}
	h_m(1,q,\ldots,q^{N-K-1})
	=\frac{(q^{m+1};q)_{N-K-1}}{(q;q)_{N-K-1}}
\end{align*}
(this is a particular case of (\ref{s_nu_q}), cf. (\ref{h_m_particular_case})). This implies that one can write the functions $\q\psi^{T}_{i}(x)$ for our $T$ as double contour integrals as follows (cf. \S \ref{sub:remark_on_contour_integral_representations_in_theorem_thm:general_q}):
\begin{align*}&
	\q\psi^{T}_{i}(x)
	=
	\frac{1}{(2\pi\i)^{2}}\oint_{\q\Cont(x)}dz
	\oint_{\cont(\infty)}\frac{dw}{w^{N+1-i}}
	\frac{(zq^{1-x};q)_{N-K-1}}{(q;q)_{N-K-1}}
	\frac{1}{w-z}\prod\nolimits_{r=1}^{N}\frac{w-q^{\nu_r-r}}{z-q^{\nu_r-r}}.
\end{align*}
Here the contour $\q\Cont(x)$ is as in Theorem \ref{thm:q_main_formula}: it encircles $q^{x},q^{x+1},\ldots,q^{\nu_1-1}$, and not $q^{x-1},q^{x-2},\ldots,q^{\nu_N-N}$; and $\cont(\infty)$ is any sufficiently big contour containing $\q\Cont(x)$.

Performing the integration over $w$ in the double contour integral above similarly to \S \ref{sub:linear_transform}, we obtain
\begin{align*}&
	\q\psi^{T}_{i}(x)
	=
	\frac{1}{2\pi\i}\oint_{\q\Cont(x)}dz
	\frac{(zq^{1-x};q)_{N-K-1}}{(q;q)_{N-K-1}}
	\prod\nolimits_{r=1}^{N}\frac1{z-q^{\nu_r-r}}
	\times\\&\hspace{140pt}
	\times
	\sum\nolimits_{j=0}^{i-1}
	z^{j}(-1)^{i-j-1}e_{i-j-1}(q^{\nu_N-N},\ldots,q^{\nu_1-1}).
\end{align*}

We again observe that the index $i$ enters $\q\psi^{T}_{i}(x)$ only through the polynomials
\begin{align*}
	\q \tilde p_i(z):=
	\sum\nolimits_{j=0}^{i-1}
	z^{j}(-1)^{i-j-1}e_{i-j-1}(q^{\nu_N-N},\ldots,q^{\nu_1-1}),
	\qquad i=1,\ldots,K.
\end{align*}
The polynomial $\q \tilde p_i(z)$ is monic of degree $i-1$ ($i=1,\ldots,K$), and thus these polynomials form a basis in $\R_{\le K-1}[z]$. Applying a suitable row transformation to the $K\times K$ matrix in (\ref{last_section_first_det_formula}), we may replace this basis with another basis:
\begin{align*}
	\q p_i(z):=
	(q;q)_{N-K}\frac{\prod_{r=1}^{N}(z-q^{-r})}
	{\prod_{r=i}^{N-K+i}(z-q^{-r})},\qquad
	i=1,\ldots,K.
\end{align*}
Clearly, $\q A_i(x\mid K,N,\nu)$ (\ref{qA_i}) is given by 
\begin{align*}
	\q A_i(x)=
	\frac{1}{2\pi\i}\oint_{\q\Cont(x)}dz
	\frac{(zq^{1-x};q)_{N-K-1}}{(q;q)_{N-K-1}}
	\q p_i(z)
	\prod\nolimits_{r=1}^{N}\frac1{z-q^{\nu_r-r}}.
\end{align*}

The linear transformation that replaces $\{\q \tilde p_i\}_{i=1}^{K}$ by $\{\q p_i\}_{i=1}^{K}$ will affect only the constant factor in (\ref{last_section_first_det_formula}). Similarly to Lemma \ref{lemma:determinant_of_p_i_coeffs}, it can be shown that the determinant of the corresponding transition matrix is equal to 
\begin{align*}
	(-1)^{K(K-1)/2}q^{K(K-1)(K-3N-2)/6}(q;q)_{N-1}\ldots (q;q)_{N-K}.
\end{align*}
Multiplying this coefficient by the factor already present in (\ref{last_section_first_det_formula}), and also by $q^{(N-K)|\ka|}$ because of the difference between (\ref{last_section_q_relative_dim}) and (\ref{last_section_first_det_formula}), after necessary simplifications we complete the proof of Theorem \ref{thm:q_main_formula}.


\subsection{Limit as $q\nearrow 1$ (proof of Proposition \ref{prop:q->1})} 
\label{sub:limit_as_qto_1_}

Fix integers $K<N$ and a signature $\nu\in\GT_N$. Write the quantities $\q A_i(x\mid K,N,\nu)$ (\ref{qA_i}) as sums of the corresponding residues:
\begin{align*}
	\q A_i(x)
	=\sum_{j\colon \nu_j-j\ge x}
	(1-q^{N-K})
	\frac{(q^{\nu_j-j+1-x};q)_{N-K-1}}
	{\prod_{r=i}^{N-K+i}(q^{\nu_j-j}-q^{-r})}
	\frac{\prod_{r=1}^{N}(q^{\nu_j-j}-q^{-r})}{
	\prod_{r\ne j}(q^{\nu_j-j}-q^{\nu_r-r})}.
\end{align*}
The $q\nearrow 1$ limit of every residue is readily computed, and we immediately see that
\begin{align*}
	\lim_{q\nearrow1}\q A_i(x\mid K,N,\nu)=(-1)^{N-K} A_i(x\mid K,N,\nu),
\end{align*}
where $A_i(x\mid K,N,\nu)$ is defined by (\ref{A_i}). This concludes the proof of Proposition~\ref{prop:q->1}.


\subsection{$q$-Gelfand--Tsetlin graph} 
\label{sub:_q_gelfand_tsetlin_graph}

In this and the next subsection we aim to explain how our formula of Theorem \ref{thm:q_main_formula} is related to the boundary of the $q$-Gelfand--Tsetlin graph. 

The $q$-Gelfand--Tsetlin graph $\q\GT$ \cite{Gorin2010q} is a branching graph which has the same vertices and edges as the ``classical'' Gelfand--Tsetlin graph described in \S \ref{sub:gelfand_tsetlin_graph} (i.e., vertices are all signatures $\GT=\bigsqcup_{N=0}^{\infty}\GT_N$, and an edge connects signatures $\mu$ and $\la$ if $\mu\prec\la$). The difference is that instead of being simple (i.e., with multiplicity~$1$), the edges of the $q$-Gelfand--Tsetlin graph carry certain \emph{formal multiplicities} depending on $q$. Namely, if $\mu\prec\la$, then we assign the multiplicity $q^{|\mu|}$ to the edge from $\mu$ to $\la$. Every increasing path in the graph of the form $\nu^{(K)}\prec \ldots\prec \nu^{(N)}$ then is also assigned a multiplicity which is defined as the product of multiplicities of the edges along this path.

It is not hard to see that the $q$-dimension $\q\Dim_N\nu$ (\ref{s_nu_q}), is equal to the sum of multiplicities of all paths from the initial vertex $\varnothing\in\GT_0$ to $\nu\in\GT_N$. The quantities $\q\Dim_{K,N}(\ka,\nu)$ (\ref{skew_dim_q}) can be interpreted in the same way if one considers paths from $\ka$ to $\nu$ (cf. \S \ref{sub:gelfand_tsetlin_graph}).


\subsection{Boundary of the $q$-Gelfand--Tsetlin graph} 
\label{sub:boundary_of_the_q_gelfand_tsetlin_graph}

The question about the boundary of the $q$-Gelfand--Tsetlin graph can be asked in the same way as in \S \ref{sub:coherent_systems} using the notion of coherent systems on the floors $\GT_N$ of the graph $\q\GT$. In the $0<q<1$ case, members $\{M_N\}$ of a coherent system (where $M_N$ is a probability measure on $\GT_N$, $N=0,1,2,\ldots$) must be compatible with the $q$-links $\q\La^{N}_{N-1}$ (\ref{qlink_La}) similarly to Definition \ref{def:coherent}. In detail, it must be 
\begin{align}\label{q_coherent_system_condition}
	\sum_{\nu\in\GT_N\colon\nu\succ\mu}
	M_N(\nu)\,q^{|\mu|}
	\frac{\q\Dim_{N-1}\mu}{\q\Dim_{N}\nu}
	=M_{N-1}(\mu),
	\quad \mbox{$\forall\, N$ and $\forall\,\mu\in\GT_{N-1}$}.
\end{align}

The boundary $\partial(\q\GT)$, i.e., the set of all extreme coherent systems on $\q\GT$ (cf.~Definition \ref{def:boundary}), was identified in \cite{Gorin2010q} with the set of all non-decreasing sequences of integers
\begin{align}\label{q_boundary}
	\mathcal{N}:=\{n_1\le n_2\le n_3\le \ldots\}\subset\Z^{\infty}.
\end{align}
It is informative to note the difference of this set with the boundary $\Om$ of the classical Gelfand--Tsetlin graph (\S \ref{sub:description_of_the_boundary}). See also \cite[end of \S 1.3]{Gorin2010q} for a brief discussion of what happens with the boundary (more precisely, with coherent systems on $\q\GT$) as $q\nearrow1$.

The problem of describing $\partial(\q\GT)$ reduces (in the same way as explained in \S \ref{sub:connection_to_question_question}) to the following question (parallel to Question \ref{question}) about asymptotics of the $q$-links: 
\begin{question}\label{q_question}
	Describe all possible sequences of signatures $\nu(1),\nu(2),\ldots$, where $\nu(N)\in\GT_N$, such that for every fixed level $K$ and signature $\ka\in\GT_K$, the sequence $\{\q\La^{N}_{K}(\nu(N),\ka)\}_{N\ge1}$ has a limit as $N$ goes to infinity. We call such sequences $\{\nu(N)\}$ \emph{$q$-regular}.
\end{question}
The result of \cite{Gorin2010q} states that $q$-regular sequences of signatures $\{\nu(N)\}$ are precisely those whose last coordinates stabilize, i.e., 
\begin{align}\label{q_regular_sequences_limit}
	\lim_{N\to\infty}\nu(N)_{N+1-j}=n_j,\qquad j=1,2,\ldots,
\end{align}
where $\boldsymbol n=\{n_1\le n_2\le \ldots\}$ is the corresponding element of the boundary $\mathcal{N}=\partial(\q\GT)$. Note that this also differs from the $q=1$ situation (see Remark \ref{rmk:reg_seq_answer}).

Since Theorem \ref{thm:q_main_formula} provides a new determinantal formula for the $q$-links, one could in principle use it to obtain the description of the boundary of $\q\GT$ in a new way similarly to what was done for the classical Gelfand--Tsetlin graph in \cite{BorodinOlsh2011GT} (see also \S \ref{sub:uniform_approximation_theorem} and \S \ref{sec:idea_of_proof_of_the_uniform_approximation_theorem}).\footnote{The original proof of \cite{Gorin2010q} is similar to the approach of \cite{OkOl1998} and is based on a Binomial Formula for certain $q$-analogues of Schur polynomials.} We do not carry out this idea in full detail, we only check that for $q$-regular sequences (\ref{q_regular_sequences_limit}) the $q$-links given by Theorem \ref{thm:q_main_formula} have a limit, and, moreover, compute it.
\begin{proposition}\label{prop:limit_q}
	Let $\{\nu(N\}$, $\nu(N)\in\GT_N$, be a $q$-regular sequence of signatures in the sense of (\ref{q_regular_sequences_limit}) corresponding to $\boldsymbol n=\{n_1\le n_2\le \ldots\}\in\mathcal{N}$. Then
	\begin{align}\label{convergence_of_skew_dim_q}
		\lim_{N\to\infty}
		\frac{\q\Dim_{K,N}\big(\ka,\nu(N)\big)}{\q\Dim_N\nu(N)}
		=\det[\q A_i(\ka_j-j\mid K,\infty,\boldsymbol n)]_{i,j=1}^{K}
	\end{align}
	for any fixed $K$ and $\ka\in\GT_K$, where 
	\begin{align}\label{qA_i_infinity}
		\q A_i(x\mid K,\infty,\boldsymbol n):=
		\frac{q^{x+K}}{2\pi\i}
		\oint_{\q\Cont(-x-K)}
		\frac{(zq^{x+K+1};q)_{\infty}(z;q)_{K-i}}
		{(z;q\mid\boldsymbol n)_{\infty}}
		dz.
	\end{align}
	Here we use the notation\footnote{Clearly, $\big(t;q\mid(0,0,\ldots)\big)_{\infty}=(t;q)_{\infty}$.}
	\begin{align*}
		(t;q)_{\infty}&:=\prod\nolimits_{r=0}^{\infty}(1-tq^{r}),\qquad
		\quad
		(t;q\mid\boldsymbol n)_{\infty}:=
		\prod\nolimits_{r=0}^{\infty}
		(1-t q^{r+n_{r+1}}),\quad \boldsymbol n\in\mathcal{N}.
	\end{align*}
\end{proposition}
\begin{proof}
	We start by investigating the behavior of the matrix elements of 
	\begin{align*}
		\det[\q A_{i}(\ka_j-j\mid K,N,\nu(N))]_{i,j=1}^{K}
	\end{align*}
	(see Theorem \ref{thm:q_main_formula}). We will assume that $K=1,2,\ldots$, $i=1,\ldots,K$, and $x\in\Z$ are fixed. Observe that the contour $\q\Cont(x)$ in the definition of $\q A_i(x\mid K,N,\nu(N))$ (\ref{qA_i}) can be replaced by $\q \Cont(x-N+K+1)$ because of the zeroes of the integrand $z=q^{x-1},q^{x-2},\ldots,q^{x-N+K+1}$ coming from the factor $(zq^{1-x};q)_{N-K-1}$ in the numerator. Let us then change the variable to $w$, $z=w q^{x-N+K+1}$, so the $w$ contour is simply $\q\Cont(0)$ which encircles the segment $[0,1]$ and not the points $q^{-1},q^{-2},\ldots$. We thus have
	\begin{align}&
		\label{qA_i_Cont_0}
		\q A_i(x\mid K,N,\nu(N))=
		\frac{(1-q^{N-K})q^{x-N+K+1}}{2\pi\i}
		\times\\&\qquad\times
		\oint_{\q\Cont(0)}dw
		\frac{(wq^{-N+K+2};q)_{N-K-1}}
		{\prod_{r=i}^{N-K+i}(wq^{x-N+K+1}-q^{-r})}
		\prod_{r=1}^{N}\frac{wq^{x-N+K+1}-q^{-r}}{wq^{x-N+K+1}-q^{\nu(N)_r-r}}.
		\nonumber
	\end{align}
	Let us transform the factors in the integrand in (\ref{qA_i_Cont_0}) one by one:
	\begin{enumerate}[\bf1.]
		\item We have
		\begin{align*}
			(wq^{-N+K+2};q)_{N-K-1}=w^{N-K-1}(-1)^{N-K-1}q^{-\binom{N-K-1}2}
			(w^{-1};q)_{N-K-1},
		\end{align*}
		and the factor $(w^{-1};q)_{N-K-1}$ tends to $(w^{-1};q)_{\infty}$ as $N\to\infty$ and $K$ is fixed.
		\item The two products of $(wq^{x-N+K+1}-q^{-r})$ in the numerator and in the denominator almost cancel out yielding $K-1$ factors:
		\begin{align*}&
			\prod\nolimits_{r=1}^{i-1}(wq^{x-N+K+1}-q^{-r})
			\prod\nolimits_{r=N-K+i+1}^{N}(wq^{x-N+K+1}-q^{-r})		
			\\&\qquad=
			q^{-N(K-1)}\prod\nolimits_{r=1}^{i-1}(wq^{x+K+1}-q^{-r+N})
			\prod\nolimits_{r=1}^{K-i}(wq^{x+K+1}-q^{r-1}).
		\end{align*}
		We have the following convergence as $N\to\infty$:
		\begin{align*}&
			\prod\nolimits_{r=1}^{i-1}(wq^{x+K+1}-q^{-r+N})
			\prod\nolimits_{r=1}^{K-i}(wq^{x+K+1}-q^{r-1})
			\\&\qquad
			\to
			\prod\nolimits_{r=1}^{i-1}(wq^{x+K+1})
			\prod\nolimits_{r=1}^{K-i}(wq^{x+K+1}-q^{r-1})
			\\&\qquad \qquad=
			w^{K-1}q^{(K-1)(x+K+1)}
			(w^{-1}q^{-x-K-1};q)_{K-i}.
		\end{align*}

		\item Finally, let us write
		\begin{align*}
			&\prod\nolimits_{r=1}^{N}\frac{1}{wq^{x-N+K+1}-q^{\nu(N)_r-r}}
			=
			\prod\nolimits_{r=1}^{N}\frac{1}{wq^{x-N+K+1}-q^{\nu(N)_{N+1-r}-N-1+r}}
			\\&\hspace{70pt}=
			w^{-N}q^{-N(x-N+K+1)}
			\prod\nolimits_{r=1}^{N}\frac{1}{1-w^{-1}q^{-x-K-2}q^{\nu(N)_{N+1-r}+r}}.
		\end{align*}
		Due to our assumption (\ref{q_regular_sequences_limit}), we have as $N\to\infty$:
		\begin{align*}
			\prod\nolimits_{r=1}^{N}\frac{1}{1-w^{-1}q^{-x-K-2}q^{\nu(N)_{N+1-r}+r}}
			\to
			\frac{1}{(w^{-1}q^{-x-K-1};q\mid \boldsymbol n)_{\infty}}.
		\end{align*}
	\end{enumerate}

	Collecting all the above transformations, we see that the integrand in (\ref{qA_i_Cont_0}) behaves as $N\to\infty$ in the following way:
	\begin{align}&
		\label{q_Ai_limit_main}
		w^{-2}\frac{(w^{-1};q)_{\infty}(w^{-1}q^{-x-K-1};q)_{K-i}}
		{(w^{-1}q^{-x-K-1};q\mid\boldsymbol n)_{\infty}}
		\times\\&
		\hspace{50pt}
		\times
		\label{q_Ai_limit_prefactor}
		(-1)^{N-K-1}
		q^{-1+Kx+K(K-1)/2}q^{N(1/2-K-x)}q^{N^2/2}.
	\end{align}

	Without the prefactor (\ref{q_Ai_limit_prefactor}), the convergence as $N\to\infty$ of the integrand is uniform on our contour $\q\Cont(0)$. Let us take this prefactor (which depends only on the column of the matrix) outside $\det[\q A_{i}(\ka_j-j\mid K,N,\nu(N))]_{i,j=1}^{K}$. Together with what was already present in (\ref{skew_dim_det_qA}), this yields a factor of $(-q^{-1})^{K}$ in front of the $K\times K$ determinant. Inserting $(-q^{-1})^{K}$ back into the determinant, we see that the desired convergence (\ref{convergence_of_skew_dim_q}) holds with 
	\begin{align*}
		\q A_i(x\mid K,\infty,\boldsymbol n)=
		\frac{-q^{-1}}{2\pi\i}
		\oint_{\q\Cont(0)}
		\frac{(w^{-1};q)_{\infty}(w^{-1}q^{-x-K-1};q)_{K-i}}
		{(w^{-1}q^{-x-K-1};q\mid\boldsymbol n)_{\infty}}
		\frac{dw}{w^{2}}.
	\end{align*}
	A change of variables $u=1/w$ gives
	\begin{align*}
		\q A_i(x\mid K,\infty,\boldsymbol n)=
		\frac{q^{-1}}{2\pi\i}
		\oint_{\q\Cont(1)}
		\frac{(u;q)_{\infty}(uq^{-x-K-1};q)_{K-i}}
		{(uq^{-x-K-1};q\mid\boldsymbol n)_{\infty}}
		du.
	\end{align*}
	Indeed, the $w$ contour $\q\Cont(0)$ encircles a segment of the form $[-\epsilon,1+\epsilon]$, so the $u$ contour contains the possible poles $q,q^{2},\ldots$, and only them, and thus can be replaced by $\q\Cont(1)$. Another change of variables, $z=uq^{-x-K-1}$, concludes the proof.	
\end{proof}


\subsection{$q$-Toeplitz matrices} 
\label{sub:_q_toeplitz_matrices}

In the ``classical'' ($q=1$) picture, for regular sequences of signatures $\{\nu(N)\}$ (see Question \ref{question}), the $N\to\infty$ limit of $A_i(x\mid K,N,\nu(N))$ (\ref{A_i}) for fixed $K$, $i$, and $x$ is equal to $\varphi_{i+x}(\om)$, where $\om$ is the point of the boundary $\partial(\GT)$ corresponding to $\{\nu(N)\}$ (see \S \ref{sub:description_of_the_boundary} and \S \ref{sec:idea_of_proof_of_the_uniform_approximation_theorem}). In other words, the limit of $A_i(-x\mid K,N,\nu(N))$ viewed as a matrix with indices $i$ and $x$, is a \emph{Toeplitz matrix}. Moreover, this limiting Toeplitz matrix is totally nonnegative (see \S \ref{sub:totally_nonnegative_toeplitz_matrices} for more discussion).

As was noted by Vadim Gorin (private communication), in the $0<q<1$ case the limiting quantities $\q A_i(x\mid K,\infty,\boldsymbol n)$ (\ref{qA_i_infinity}) should satisfy some version of the \emph{$q$-Toeplitz property} introduced in \cite{Gorin2010q}. The integral formula (\ref{qA_i_infinity}) allows to observe such a property directly:
\begin{proposition}\label{prop:qA_i_qToeplitz}
	For any boundary point $\boldsymbol n\in\mathcal{N}=\partial({\q\GT})$, all fixed $K=1,2,\ldots$, $i=1,\ldots,K$, and $x\in\Z$, one has
	\begin{align}\label{three_term_qA_i}
		\q A_{i-1}(x)\cdot q^{i}=\q A_{i}(x-1)\cdot 
		q^{1-x}+\q A_{i}(x)\cdot (q^{i}-q^{-x}),
	\end{align}
	where we abbreviate $\q A_{j}(y)\equiv \q A_j(y\mid K,\infty,\boldsymbol n)$.
\end{proposition}
Note that for $q=1$, (\ref{three_term_qA_i}) is reduced to $\q A_{i-1}(x)=\q A_{i}(x-1)$, which agrees with the usual Toeplitz property in the $q\nearrow1$ limit (cf. \S \ref{sub:limit_as_qto_1_}).
\begin{proof}
	Due to the zeroes of the integrand, for $\q A_{i}(x-1)$ the contour $\q\Cont(-x+1-K)$ in (\ref{qA_i_infinity}) can be replaced by $\q\Cont(-x-K)$. In this way, all the three terms in (\ref{three_term_qA_i}) are expressed as integrals over the same contour. Then it is readily checked that the desired three-term relation is satisfied by the corresponding integrands.
\end{proof}

To rewrite relation (\ref{three_term_qA_i}) exactly in the form of the $q$-Toeplitz property \cite[(5)]{Gorin2010q}, introduce new quantities
\begin{align}\label{B_x_i}
	B^{\boldsymbol n}(x,i):=
	\q A_{K+1-i}(x-K-1\mid K,\infty,\boldsymbol n)
	\cdot q^{\frac12(x-i)(x+i-3)}.
\end{align}
It readily follows from (\ref{qA_i_infinity}) that $B^{\boldsymbol n}(x,i)$ does not depend on $K$:
\begin{align}\label{B_x_i_cont}
	B^{\boldsymbol n}(x,i)=
	\frac{q^{\frac12(x-i+1)(x+i-2)}}{2\pi\i}
	\oint_{\q\Cont(-x+1)}
	\frac{(zq^{x};q)_{\infty}(z;q)_{i-1}}
	{(z;q\mid\boldsymbol n)_{\infty}}
	dz.
\end{align}
For these $B^{\boldsymbol n}$'s it can be checked that
\begin{align}\label{q_Toeplitz_of_B}
	B^{\boldsymbol n}(x,i+1)=B^{\boldsymbol n}(x-1,i)+(q^{1-i}-q^{1-x})B^{\boldsymbol n}(x,i),
\end{align}
which coincides with \cite[(5)]{Gorin2010q}.

The singe contour integral formula for $B^{\boldsymbol n}(x,i)$ allows to observe one more property of these quantities:
\begin{proposition}\label{prop:B_1_summation}
	We have for $n_1\ge0$:\footnote{This assumption is not very restrictive, see \cite[Thm.~1.1.3]{Gorin2010q}.}
	\begin{align}\label{B_1_summation}
		\sum_{\ell=0}^{\infty}
		B^{\boldsymbol n}(\ell+1,1)\prod_{i=0}^{\ell-1}(q^{-i}-z)=
		\frac{(z;q)_{\infty}}{(z;q\mid \boldsymbol n)_{\infty}}.
	\end{align}
	The right-hand side is an entire function in $z$ (by the Weierstrass factorization theorem), and the series converges everywhere in $\C$.
\end{proposition}
This proposition follows from a more general lemma:
\begin{lemma}\label{lemma:q_Cauchy}
	Let $\phi(z)$ be an entire function, and consider the expansion
	\begin{align*}
		\phi(z)=\sum_{\ell=0}^{\infty}
		c_\ell\prod_{i=0}^{\ell-1}(q^{-i}-z).
	\end{align*}
	The coefficients of this expansion admit the following integral representation:
	\begin{align*}
		c_{\ell}=
		\frac{q^{\ell(\ell+1)/2}}{2\pi\i}
		\oint_{\q\Cont(-\ell)}
		\frac{\phi(z)}{(z;q)_{\ell+1}}dz,\qquad \ell=0,1,2,\ldots.
	\end{align*}
\end{lemma}
\begin{remark}
	Lemma \ref{lemma:q_Cauchy} may be viewed as an inversion formula for some $q$-Laplace transform. A similar statement appears in \cite[Prop. 3.1.1]{BorodinCorwin2011Macdonald} with references to $q$-versions of the Laplace transform in \cite{Hahn1949} and in a recent manuscript \cite{Bangerezako}. To make our argument self-contained, let us present a proof of this statement.
\end{remark}
\noindent\emph{Proof of Lemma \ref{lemma:q_Cauchy}.} 
	Fix $j\ge0$ and consider
	\begin{align*}
		\frac{\phi(z)}{(z;q)_{j+1}}&=
		\frac{c_{j}q^{-j(j-1)/2}}{1-zq^{j}}+
		\sum_{\ell=0}^{j-1}c_\ell q^{-\ell(\ell-1)/2}
		\frac{(z;q)_{\ell}}{(z;q)_{j+1}}+
		\mbox{an entire function}.
	\end{align*}
	Integrating this equality over the contour $\q\Cont(-j)$ which encircles only the possible poles $q^{-j},q^{-j+1},\ldots$, we see that the contribution from the sum over $\ell$ vanishes because every term of this sum has no residue outside $\q\Cont(-j)$ (it behaves as $\mathrm{const}\cdot z^{\ell-j-1}$ at infinity). The holomorphic part also vanishes, so we have 
	\begin{align*}
		\frac{1}{2\pi\i}
		\oint_{\q\Cont(-j)}
		\frac{\phi(z)}{(z;q)_{j+1}}dz=
		\frac{1}{2\pi\i}
		\oint_{\q\Cont(-j)}
		\frac{c_{j}q^{-j(j-1)/2}}{1-zq^{j}}dz.
	\end{align*}
	Computing the integral in the right-hand side, we conclude the proof.
\qed

\smallskip

\noindent\emph{Proof of Proposition \ref{prop:B_1_summation}.} From (\ref{B_x_i_cont}) we have 
\begin{align*}
	B^{\boldsymbol n}(\ell+1,1)=
	\frac{q^{\frac12\ell(\ell+1)}}{2\pi\i}
	\oint_{\q\Cont(-\ell)}
	\frac{(zq^{\ell+1};q)_{\infty}}
	{(z;q\mid\boldsymbol n)_{\infty}}
	dz.
\end{align*}
It remains to note that ${(zq^{\ell+1};q)_{\infty}}=\frac{(z;q)_{\infty}}{(z;q)_{\ell+1}}$, and use Lemma \ref{lemma:q_Cauchy}. \qed
\medskip

Thus, we are led to the following statement:
\begin{proposition}\label{prop:limit_q_links_q_Toeplitz}
	For a $q$-regular sequence $\{\nu(N)\}$ corresponding to a boundary point $\boldsymbol n\in\mathcal{N}=\partial(\q\GT)$ with $n_1\ge0$, the limit of the $q$-links is given by 
	\begin{align}&
		\label{limit_q_links_q_Toeplitz}
		\lim_{N\to\infty}
		\q\La^{N}_{K}\big(\nu(N),\ka\big)
		=\Dim_K\ka\cdot
		q^{-(K-1)|\ka|}q^{\sum_{r=1}^{K}(r\ka_r-\ka_r(\ka_r+1)/2)}
		\times\\&\hspace{190pt}\times
		\det[B^{\boldsymbol n}(\ka_{K+1-i}+i,j)]_{i,j=1}^{K},
		\nonumber
	\end{align}
	where $B^{\boldsymbol n}(i,j)$, $i,j\ge1$, is a unique $q$-Toeplitz matrix (in the sense of (\ref{q_Toeplitz_of_B})) whose first column satisfies (\ref{B_1_summation}).
\end{proposition}
\begin{proof}
	This a combination of (\ref{qlink_La}), Propositions \ref{prop:limit_q} and \ref{prop:B_1_summation} and the fact that a $q$-Toeplitz matrix is completely determined by its first column via (\ref{q_Toeplitz_of_B}). In the right-hand side of (\ref{limit_q_links_q_Toeplitz}) we have also rewritten the determinant $\det[\q A_i(\ka_j-j\mid K,\infty,\boldsymbol n)]_{i,j=1}^{K}$ in terms of the $q$-Toeplitz elements $B^{\boldsymbol n}(i,j)$.
\end{proof}

\begin{remark}
	\textbf{1.} The limit $\lim_{N\to\infty}\q\La^{N}_{K}\big(\nu(N),\ka\big)$ in the left-hand side of (\ref{limit_q_links_q_Toeplitz}) is in fact (see the discussion of \S \ref{sub:connection_to_question_question} which is also applicable in the $0<q<1$ case) equal to the value of the extreme coherent system corresponding to $\boldsymbol n\in\mathcal{N}$ at the signature $\ka\in\GT_K$. This quantity is denoted by $\mathcal{E}^{\boldsymbol n}_{K}(\ka)$ in \cite{Gorin2010q}.

	\textbf{2.}	Proposition \ref{prop:limit_q_links_q_Toeplitz} can be deduced from the results of \cite{Gorin2010q} in the following way. Under the assumption $n_1\ge0$, in \cite[Thm.~1.1.2]{Gorin2010q} a certain generating function for the quantities $\{\mathcal{E}^{\boldsymbol n}_{K}(\ka)\}_{\ka\in\GT_K}$ is written out explicitly as a $K$-fold product.\footnote{This development is parallel to the $q=1$ considerations explained in \S \ref{sub:totally_nonnegative_toeplitz_matrices}.} This fact (together with some formulas from \cite[proofs of Prop.~1.4~(\S7) and Lemma~6.5 (\S6.2)]{Gorin2010q}) allows to write the identity (\ref{limit_q_links_q_Toeplitz}). The elements of the $q$-Toeplitz matrix there are defined in the same way as in Proposition \ref{prop:limit_q_links_q_Toeplitz} (e.g., see \cite[Thm. 7.1]{Gorin2010q}).

	This implies that the minors of the $q$-Toeplitz matrix $B^{\boldsymbol n}(i,j)$, $i,j\ge1$, that enter the right-hand side of (\ref{limit_q_links_q_Toeplitz}) for various $K$ and $\ka\in\GT_K$ are nonnegative. However, the matrix $B^{\boldsymbol n}(i,j)$ itself is not totally nonnegative. See \cite[\S1.5]{Gorin2010q} for more discussion.
\end{remark}

In addition to reproving some results of \cite{Gorin2010q}, using our contour integral formulas we are able to readily obtain a solution of the $q$-Toeplitz recurrence relation 
\begin{align}\label{q_Toeplitz_equation}
	d(x,i+1)=d(x-1,i)+(q^{1-i}-q^{1-x})d(x,i), \qquad
	x,i=1,2,\ldots
\end{align}
(with agreement that $d(x,i)=0$ if either $x$ or $i$ is $\le0$) with initial condition 
\begin{align}\label{q_Toeplitz_IC}
	d(\ell+1,1)=c_{\ell},\qquad \ell=0,1,\ldots.
\end{align}

\begin{proposition}
	Let the series 
	\begin{align*}
		\phi(z)=\sum_{\ell=0}^{\infty}
		c_\ell\prod_{i=0}^{\ell-1}(q^{-i}-z)
	\end{align*}
	converge everywhere in $\C$. Then the solution of (\ref{q_Toeplitz_equation})--(\ref{q_Toeplitz_IC}) is given for $x,i\ge1$ by the following contour integral:
	\begin{align*}
		d(x,i)=\frac{q^{\frac12(x-i+1)(x+i-2)}}{2\pi\i}
		\oint_{\q\Cont(-x+1)}
		\phi(z)
		\frac{(z;q)_{i-1}}
		{(z;q)_{x}}
		dz.
	\end{align*}
\end{proposition}



\providecommand{\bysame}{\leavevmode\hbox to3em{\hrulefill}\thinspace}
\providecommand{\MR}{\relax\ifhmode\unskip\space\fi MR }
\providecommand{\MRhref}[2]{%
  \href{http://www.ams.org/mathscinet-getitem?mr=#1}{#2}
}
\providecommand{\href}[2]{#2}

\end{document}